\documentclass[a4paper,11pt,reqno,
final
]{amsart}	

\usepackage{amsmath}
\usepackage{amssymb}
\usepackage{amsthm}
\usepackage{graphicx}
\usepackage{dcolumn}
\usepackage{bm}
\usepackage{psfrag}      
\usepackage{subfig}
\usepackage{pstricks}
\usepackage{setspace}  

\newtheorem{theorem}{Theorem}[section]
\newtheorem{lemma}[theorem]{Lemma}
\newtheorem{proposition}[theorem]{Proposition}
\newtheorem{corollary}[theorem]{Corollary}
\newenvironment{definition}[1][Definition]{\begin{trivlist}
\item[\hskip \labelsep {\bfseries #1}]}{\end{trivlist}}

\newenvironment{remark}[1][Remark]{\begin{trivlist}
\item[\hskip \labelsep {\bfseries #1}]}{\end{trivlist}}

\newcommand{\sd}{\mathop{\rm sd}}

\newcommand{\sn}{\mathop{\rm sn}}
\newcommand{\dn}{\mathop{\rm dn}}

\newcommand{\sign}{\mathop{\rm sign}}

\def\R{{\mathbb R}}

\def\N{\mathbb{N}}

\def\rmd{{\mathrm d}}

\def\rmi{{\mathrm i}}

\begin{document}

\title[An asymptotic universal focal decomposition]{An asymptotic universal focal decomposition for non-isochronous potentials}

\author[C. A. A. de Carvalho]{C. A. A. de Carvalho}
\address[C. A. A. de Carvalho]{Universidade Federal do Rio de Janeiro, Rio de Janeiro, Brazil}
\email{aragao@if.ufrj.br}

\author[M. M. Peixoto]{M. M. Peixoto}
\address[M. M. Peixoto]{Instituto de Matem\'atica Pura e Aplicada, Rio de Janeiro, Brazil}
\email{peixoto@impa.br}

\author[D. Pinheiro]{D. Pinheiro}
\address[D. Pinheiro]{CEMAPRE, ISEG-Universidade T\'ecnica de Lisboa, Lisboa, Portugal}
\email{dpinheiro@iseg.utl.pt}

\author[A. A. Pinto]{A. A. Pinto}
\address[A. A. Pinto]{Departamento de Matem\'atica, Universidade do Minho, Braga, Portugal}
\email{aapinto@math.uminho.pt}

\begin{abstract}
Galileo, in the XVII century, observed that the small oscillations of a pendulum seem to have constant period. In fact, the Taylor expansion of the period map of the pendulum is constant up to second order in the initial angular velocity around the stable equilibrium. It is well known that, for small oscillations of the pendulum and small intervals of time, the dynamics of the pendulum can be approximated by the dynamics of the harmonic oscillator. We study the dynamics of a family of mechanical systems that includes the pendulum at small neighbourhoods of the equilibrium but after long intervals of time so that the second order term of the period map can no longer be neglected. We analyze such dynamical behaviour through a renormalization scheme acting on the dynamics of this family of mechanical systems. The main theorem states that the asymptotic limit of this renormalization scheme is universal: it is the same for all the elements in the considered class of mechanical systems. As a consequence, we obtain an universal asymptotic focal decomposition for this family of mechanical systems. This paper is intended to be the first of a series of articles aiming at a semiclassical quantization of systems of the pendulum type as a natural application of the focal decomposition associated to the two-point boundary value problem.
\end{abstract}


\maketitle

\section{Introduction} 

\subsection{Focal decomposition}

The concept of focal decomposition was introduced by Peixoto in \cite{Peixoto} (under the name of $\sigma$-decomposition), and further developed by Peixoto and Thom in \cite{Peixoto_Thom}. The starting point was the 2-point boundary value problem for ordinary differential equations of the second order
\begin{eqnarray}\label{2BVP}
&&\ddot{x}=f(t,x,\dot{x}) \ , t,x,\dot{x}\in\R \nonumber \\
&&x(t_1)= x_1 \ , \; x(t_2)=x_2  \ ,
\end{eqnarray}
which was formulated for the equation of Euler at the beginnings of the calculus of variations in the first half of the 18th Century. This is the simplest and oldest of all boundary value problems. Accordingly, there is a vast literature about it, mostly in the context of applied mathematics where one frequently uses the methods of functional analysis. Here we adopt a different point of view: we look for the number $i\in\{0, 1, 2,...,\infty\} = \N$ of solutions of problem \eqref{2BVP}, and how this number varies with the endpoints $(t_1, x_1)$ and $(t_2, x_2)$.

Let $\R^4= \R^2(t_1,x_1)\times\R^2(t_2,x_2)$ be the set of all pairs of points of the $(t, x)$-plane and to each point $(t_1,x_1,t_2,x_2)\in\R^4$ associate the number of solutions $i$ of the boundary value problem \eqref{2BVP}. When $t_1=t_2$, the index $i$ is defined as being 0 if $x_1\ne x_2$, and $\infty$ if $x_1=x_2$. The diagonals $\delta$, $\Delta$ are defined as
\begin{align*}
\delta &= \left\{(t_1,x_1,t_2,x_2)\in\R^4: t_1=t_2\right\} \\
\Delta &= \left\{(t_1,x_1,t_2,x_2)\in\R^4: t_1=t_2 , \; x_1=x_2 \right\} \ .
\end{align*}
Let $\Sigma_i\subset \R^4$ be the set of points to which the index $i$ has been assigned. Clearly $\R^4$ is the disjoint union of all the sets $\Sigma_i$, that is, these $\Sigma_i$ define a partition of $\R^4$. This partition is called the focal decomposition of $\R^4$ associated with the boundary value problem \eqref{2BVP}:
\begin{align*}
\R^4=\Sigma_0\cup\Sigma_1\cup ...\cup\Sigma_\infty\ .
\end{align*}
The decomposition above is the central object of study for a given $f$.

From the above definition, the diagonal $\delta$ is 3-dimensional and is naturally decomposed into the 2-dimensional connected component $\Delta$ belonging to $\Sigma_\infty$ and two 3-dimensional connected components belonging to $\Sigma_0$. This is called the natural stratification of $\delta$.

If one of the endpoints in \eqref{2BVP} is kept fixed, say $(t_1, x_1)=(0, 0)$, then the sets $\Sigma_i$ induce a decomposition of $\R^2(t_2, x_2)$ by the sets $\sigma_i=\Sigma_i\cap\left(\left\{\left(0,0\right)\right\}\times\R^2(t_2,x_2) \right)$. The restricted problem with base point $(0, 0)$ consists of finding the corresponding focal decomposition of $\R^2$ by the sets $\sigma_i$:
\begin{align*}
\R^2=\sigma_0\cup\sigma_1\cup ...\cup\sigma_\infty\ .
\end{align*}
Throughout this paper, when we speak
of the focal decomposition associated to problem \eqref{2BVP}, we will be referring to the focal decomposition of $\R^2$ defined above.

One has some clarification of what the sets $\sigma_i$, $\Sigma_i$ are if one relates them to a certain surface associated to each point $(t,x)$ called the \textit{star} of this point. It is the \textit{lifted manifold} introduced in \cite{Peixoto_4} for other purposes. Let
\begin{eqnarray} \label{ueq}
\frac{dt}{dt}=1\,\,\, , \frac{dx}{dt}=u\,\,\, , \frac{du}{dt}=f(t,x,u)
\end{eqnarray}
be the first system in $\R^3(t,x,u)$ equivalent to \eqref{2BVP}. One calls the \textit{star} associated to the base point $(t,x)$ the surface $S(t,x)\subset \R^3(t,x,u)$ obtained by the union of the trajectories of \eqref{ueq} passing through the points of the line $(t,x,u), -\infty<u<\infty$. One verifies immediately that if $\pi$ is the projection $\pi(t,x,u)=(t,x)$ then: a) with respect to base point $(t_1, x_1)$, $(t_2, x_2) \in \sigma_i$ iff $[\pi\mid S(t_1,x_1)]^{-1} (t_2,x_2)$ consists of $i$ points; also b) $(t_1,x_1,t_2,x_2) \in \Sigma_i$ iff $S(t_1,x_1) \cap S(t_2,x_2)$ consists of $i$ solutions of \eqref{ueq}. See \cite{Peixoto_Thom,Peixoto_Silva} for more details on the relationship between the focal decomposition and stratification theory.

Concerning the star, a final comment seems appropriate. To each curve $\gamma$ solution of a two point boundary value problem, we associate two stars, one at each endpoint, which do intersect. When they do so transversally, we say that $\gamma$ is transversal or ``structurally stable'' because $\gamma$ can not be perturbed away by small variations of the equation. Besides, this situation is generic \cite{Peixoto_4}, i.e. any equation can be approximated by one for which the solutions of the $2$-point boundary problem are ``structurally stable''.

A notorious example of a focal decomposition due to Peixoto and Thom \cite{Peixoto_Thom}, is provided by the focal decomposition of the pendulum equation $\ddot{x}+\sin x=0$ with base point $(0,0)$ (see Figure \ref{UFD_pend_fig}). This focal decomposition contains non-empty sets $\sigma_i$ with all finite indices. Every set $\sigma_{2k-1}$, $k=1,2,...$, consists of a 2-dimensional open set plus the cusp-point $(\pm k\pi,0)$; they all have two connected components. All four connected components of the even-indexed sets $\sigma_{2k}$ are open-arcs, asymptotic to one of the lines $x=\pm\pi$ and incident to the cusp-points $(\pm k\pi,0)$; the lines $x=\pm\pi$ are part of $\sigma_1$, except for the points $(0,\pm \pi)$ which belong to $\sigma_0$.

\begin{figure}[h!]
	\centering
      \psfrag{t}[cc][][0.85][0]{$t$}%
      \psfrag{x}[cc][][0.85][0]{$x$}
      \psfrag{p1}[cc][][0.85][0]{$\pi$}
      \psfrag{p2}[cc][][0.85][0]{$2\pi$}
      \psfrag{p3}[cc][][0.85][0]{$3\pi$}
      \psfrag{-p1}[cc][][0.85][0]{$-\pi$}
      \psfrag{-p2}[cc][][0.85][0]{$-2\pi$}
      \psfrag{-p3}[cc][][0.85][0]{$-3\pi$}
      \psfrag{si}[cc][][0.85][0]{$\sigma_\infty$}
      \psfrag{s0}[cc][][0.85][0]{$\sigma_0$}
      \psfrag{s1}[cc][][0.85][0]{$\sigma_1$}
      \psfrag{s2}[cc][][0.85][0]{$\sigma_2$}
      \psfrag{s3}[cc][][0.85][0]{$\sigma_3$}
      \psfrag{s4}[cc][][0.85][0]{$\sigma_4$}
      \psfrag{s5}[cc][][0.85][0]{$\sigma_5$}
      \psfrag{s6}[cc][][0.85][0]{$\sigma_6$}
      \psfrag{s7}[cc][][0.85][0]{$\sigma_7$}
      \psfrag{s8}[cc][][0.85][0]{$\sigma_8$}
      \psfrag{s9}[cc][][0.85][0]{$\sigma_9$}
      \psfrag{x=p}[cc][][0.85][0]{$x=\pi$}
      \psfrag{x=-p}[cc][][0.85][0]{$x=-\pi$}
   		\includegraphics[width=80mm,angle=-90]{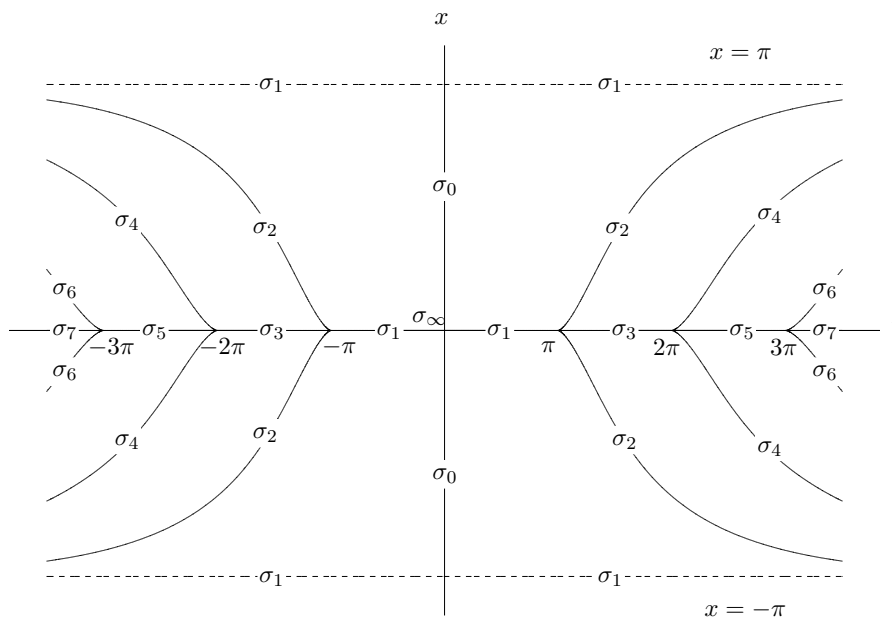}
   		\caption{The pendulum's focal decomposition.}
   		\label{UFD_pend_fig}
\end{figure}

There are many examples in the calculus of variations where rules are given for determining the number of solutions of the boundary value problem \eqref{2BVP} and for which the corresponding focal decomposition can be derived. The first such example seems to be a theorem of Goldschmidt (1831) concerning the number of catenaries passing through two points in the plane \cite[p. 340]{Todhunter}. See also the papers by MacNeish \cite{MacNeish} and Peixoto \cite{Peixoto_4}, and especially the textbook by Collatz \cite[pp. 170ff.]{Collatz}, where several focal decompositions are given. However, the number of distinct indices obtained for the focal decompositions in these examples is very small and the concept itself is not recognized.

After its recognition and early development in \cite{Peixoto,Peixoto_Thom}, the concept of focal decomposition was further developed by Kupka and Peixoto in \cite{Kupka_Peixoto} in the context of geodesics and by Peixoto and Silva in \cite{Peixoto_Silva} regarding the Euler equation associated with a given Lagrangian function $L$. In \cite{Kupka_Peixoto_Pugh}, Kupka, Peixoto and Pugh proved generic results pointing out the existence of residual sets in the space of complete Riemannian metrics on a finite dimensional complete Riemannian manifold $M$, for which bounds upon the number of geodesics of prescribed length on $M$ were obtained. Relationships with the arithmetic of positive definite quadratic forms are considered by Kupka and Peixoto in \cite{Kupka_Peixoto} and by Peixoto in \cite{Peixoto_2,Peixoto_3}, and in \cite{Kupka_Peixoto,Peixoto_2} attention is also drawn to the relationship between focal decomposition and the Brillouin zones of solid state physics, a view that is further developed later by Veerman et al. in  \cite{Veerman_Peixoto_Rocha_Sutherland}. In \cite{Peixoto_3}, it is pointed out that the focal decomposition associated to \eqref{2BVP} is relevant for the computation of the semiclassical quantization of this equation via the Feynman path integral method and in  \cite{Carvalho_Cavalcanti_Fraga_Joras_1,Carvalho_Cavalcanti_Fraga_Joras_2} de Carvalho et al. exhibit further relations with quantum statistical mechanics.

The concept of focal decomposition might also be relevant to the study of caustic formation by focusing wavefronts. This appears in optics (see the paper by Berry and Upstill \cite{Berry_Upstill}), tsunami formation (see the papers by Berry \cite{Berry_1,Berry_2}) or general relativity (see the papers by Friedrich and Stewart \cite{Friedrich_Stewart}, Hasse et al. \cite{Hasse_Kriele_Perlick}, Ellis et al. \cite{Ellis_Bassett_Dunsby} and Ehlers and Newman \cite{Ehlers_Newman}). It is not so surprising that focal decomposition, i.e. the two-point boundary-value problem, is intimately connected to the very formulation of a variational problem, as Euler taught us.

\subsection{Renormalization}

Besides focal decomposition, the other key topic in this paper is renormalization --- the study of asymptotic self-similarity. The main idea behind renormalization is the introduction of an operator --- the renormalization operator --- on a space of systems whose action on each system is to remove its small scale behaviour and to rescale the remaining variables to preserve some normalization. If a system converges to some limiting behaviour under iteration of the renormalization operator then we say that such behaviour is universal. Since the renormalization operator relates different scales, such universal behaviour is self-similar.

In this paper we introduce a new kind of renormalization scheme acting on a family of mechanical systems which includes the pendulum. \emph{This scheme has the distinguishing feature that time is not rescaled, but rather translated, while the initial velocities and space are appropriately scaled.}

Renormalization is extremely relevant in several areas of physics. The first ideas concerning renormalization were introduced in the 1940's in the context of quantum electrodynamics by Bethe, Feynman, Schwinger and Dyson. These ideas were later developed by Stueckelberg and Petermann in \cite{Stueckelberg_Petermann} and by Gell-Mann and Low in \cite{Gell-Mann_Low}, with the introduction of the Renormalization Group as a tool to improve approximate solutions to quantum field theory equations. Later developments in the subject were due to Callan \cite{Callan}, Symanzik \cite{Symanzik} and Weinberg \cite{Weinberg}, among many others. Kadanoff in \cite{Kadanoff} and Wilson in \cite{Wilson} introduced renormalization techniques into statistical mechanics in order to improve the understanding of critical phenomena.

The relevance of renormalization in dynamical systems was first noticed by Feigenbaum in \cite{Feigenbaum_1,Feigenbaum_2} and, independently, by Coullet and Tresser in \cite{Coullet_Tresser_1}, with their discovery of period doubling universality in unimodal maps of the interval. They introduced a renormalization operator --- the period doubling operator --- to show that period doubling sequences for this class of maps are asymptotically self-similar and that these sequences have an identical form for a large open set of such class of maps. This renormalization operator consists of time and space asymptotic normalizations chosen to preserve the dynamical characteristics of the maps under iteration. Infinitely renormalizable maps separate the regions with regular (zero entropy) and chaotic (positive entropy) dynamics and, moreover, their invariant set is universal, i.e. it is the same for all infinitely renormalizable unimodal maps. It was soon realized that the period-doubling operator was just a restriction of another operator acting on the space of unimodal maps –-- the renormalization operator –-- whose dynamical behaviour is much richer.

Renormalization has been a very active research area in dynamical systems in the past two decades:  Sullivan in \cite{Sullivan_1,Sullivan_2}, McMullen in \cite{McMullen}, de Melo and Pinto in \cite{Melo_Pinto}, Lyubich in \cite{Lyubich} and Faria, de Melo and Pinto in \cite{Faria_Melo_Pinto} studied the renormalization operator acting in families of unimodal maps. The renormalization operator naturally appears in several other families of maps, such as families of critical circle maps or families of annulus maps (see, for instance, MacKay \cite{MacKay_thesis,MacKay}, Ostlund et al. \cite{Ostlund_Rand_Sethna_Siggia}, Lanford \cite{Lanford_1,Lanford_2}, de Melo \cite{Melo}, Martens \cite{Martens}, de Faria and de Melo \cite{Faria_Melo_1,Faria_Melo_2} and Yampolsky \cite{Yampolsky_1,Yampolsky_2}).

The main subject of this paper is a renormalization scheme acting on the dynamics of a family of mechanical systems that includes the pendulum. Our motivation for the introduction of such a scheme comes from the restricted focal decomposition with base point $(0,0)$ of the pendulum equation $\ddot{x}+\sin x = 0$ in Figure \ref{UFD_pend_fig}. It turns out that the sequence formed by the even-indexed sets in the pendulum's focal decomposition is approximately self-similar. The renormalization scheme we introduce can then be justified in the following way: for a large integer $n$, we consider the even-indexed set $\sigma_{2n}$ and, contrary to previous renormalizations, we do not rescale time but just shift it so that its origin is at $t=n\pi$. We then restrict the initial velocities to a small interval so that that the index corresponding to the shifted even-indexed set is equal to one; we complete the procedure by normalizing space in such way a that the shifted even-indexed set is asymptotic to the lines $x=\pm 1$. Under iteration of this renormalization scheme, we obtain asymptotic trajectories that define an asymptotic focal decomposition. Both the asymptotic trajectories and focal decomposition are universal and self-similar.

To be more precise, our renormalization scheme acts on the dynamics of a family of mechanical systems (see the books \cite{Abraham_Marsden,Arnold,Mackay_Meiss,Marsden_Ratiu} for details on mechanical systems) defined by a Lagrangian function $L(x,\dot{x})= \dot{x}^2/2-V(x)$, where $V(x)$ is a given non-isochronous potential, i.e. not all the periodic solutions of the corresponding Euler--Lagrange equation have the same period. For more information on isochronous potentials, see the paper \cite{Bolotin_MacKay_2} by Bolotin and MacKay and references therein. We obtain that the asymptotic limit of this renormalization scheme is universal and as a consequence we obtain an universal asymptotic focal decomposition for this family of mechanical systems. 

This paper is the first step towards a broader research program, proposed by Peixoto and Pinto, connecting renormalization techniques, focal decomposition of differential equations and semiclassical physics. The next steps of this research program include a prove of the convergence of renormalized focal decompositions to the universal asymptotic focal decomposition and an extension of this renormalization scheme to obtain a 4-dimensional universal asymptotic focal decomposition, i.e. with no restrictions on the base point of the boundary value problem. The ultimate goal of such program would be to deal with applications of this renormalization procedure to semiclassical physics.

The present paper is structured in the following way: in section \ref{semiclassical}, we show why we believe that our program is relevant for semiclassical physics; in section \ref{intro_results}, we introduce the main concepts we deal with throughout the paper and rigorously state our main results, while section \ref{Ren_sect} is devoted to providing the reader with a global picture of the strategy of the proof  of the main results. The remaining sections are more technical: section \ref{QP} deals with mechanical systems defined by some \emph{quartic potentials} -- we compute explicit solutions of the Euler--Lagrange equation associated with such mechanical systems in terms of Jacobian elliptic functions (see the textbooks \cite{Abramowitz_Stegun,Bowman,Whittaker_Watson} for further details on Jacobian elliptic functions) and use known estimates for such functions to obtain estimates for our explicit solutions; section \ref{Perturbed_Potentials} is devoted to the study of mechanical systems defined by perturbations of the \emph{quartic potentials} -- we use an implicit function argument inspired in a paper by Bishnani and MacKay \cite{Bishnani_MacKay} to obtain estimates for the solutions of this second class of mechanical systems as a continuation of the estimates obtained previously for the explicit solutions associated with the \emph{quartic potentials}; section \ref{period_map} is devoted to the study of the period map for the periodic orbits of this family of mechanical systems; sections \ref{scaling_parameter} and \ref{trajectories} are devoted to the renormalization scheme introduced in this paper. We summarize in section \ref{conclusion}.

\section{Semiclassical physics}\label{semiclassical}
Focal decomposition is in fact a first step towards semiclassical quantization. This was already recognized in the semiclassical calculation of partition functions for quantum mechanical systems, where the need to consider a varying number of classical solutions in different temperature regimes became evident \cite{Carvalho_Cavalcanti}. 

We shall illustrate this for two physical quantities which can be expressed in terms of Feynman path integrals: in quantum mechanics, we shall consider the one-dimensional propagator between an initial point $x_1$, at time $t_1$, and a final point $x_2$, at time $t_2$; in quantum statistical mechanics, we shall study the one-dimensional thermal density matrix element
in position representation. For further details on semiclassical quantization, see the textbook \cite{Feynman_Hibbs} by Feynman and Hibbs or the review paper \cite{DeWitt-Morette} by DeWitt-Morette.

In the quantum mechanical case, the propagator is just the time evolution operator $\exp[-\rmi H(t_2-t_1)/\hbar]$ computed between position eigenstates $|x_1\rangle$ and $|x_2\rangle$
\begin{equation*}
G(x_1, t_1; x_2, t_2) = \langle x_2| \exp[-\rmi H(t_2- t_1)/\hbar]|x_1\rangle \ ,
\end{equation*}
where $H$ is the Hamiltonian
\begin{equation*}
H = \frac{p^2}{2m} + V (x) \ .
\end{equation*}
The propagator may be written as a path integral
\begin{equation}\label{SC1}
G(x_1, t_1; x_2, t_2) = \int^{z(t_2)=x_2}_{z(t_1)=x_1}[Dz(t)]\exp\left(\rmi \frac{S[z]}{\hbar}\right) \ ,
\end{equation}
where $S$ is the classical action of the mechanical problem
\begin{equation*}
S[z] = \int^{t_2}_{t_1} \left[\frac{1}{2}m\dot{z}^2 - V (z)\right] \rmd t\ .
\end{equation*}
The integral over $z$ in \eqref{SC1} stands for a sum over all trajectories that connect $x_1$ at $t_1$ to $x_2$ at $t_2$, an object whose mathematical characterization has led to much investigation over the years.

We shall only be interested in the leading semiclassical expression for the propagator, which can be formally derived from \eqref{SC1}. The resulting expression is
\begin{equation*}
G_{sc}(x_1, t_1; x_2, t_2) =\sum_{n=1}^{i} \left\{\det[S^{(2)}_n]\right\}^{-1/2}\exp\left(\rmi \frac{S_n}{\hbar}\right) \ .
\end{equation*}
The $i$ in the upper limit is the same as in the focal decomposition, and indicates that the approximation is restricted to all classical trajectories, i.e. all solutions of the classical equation of motion which satisfy the boundary conditions $z(t_1) = x_1$ and $z(t_2) = x_2$. The $S_n$ in the exponential stands for the value of the action of the $n^{\text{th}}$ classical trajectory, whereas $S^{(2)}_n$ denotes the second functional derivative of the action with respect to $z(t)$, computed at the $n^{\text{th}}$ classical trajectory. The inverse square root of the determinant
of that operator, the so-called van Vleck determinant, accounts for the first quantum corrections in a semiclassical expansion.

In the case of quantum statistical mechanics, the thermal density matrix element is just the Boltzmann operator $\exp[-\beta H]$ computed between position eigenstates $|x_1\rangle$ and $|x_2\rangle$, i.e. $\rho(x_2, x_1) = \langle x_2| \exp[-\beta H |x_1\rangle$. The path integral for this quantity is given by
\begin{equation*}
\rho(x_2, x_1) = \int^{z(\beta\hbar)=x_2}_{z(0)=x_1}[Dz(\tau)]\exp\left(-\frac{S[z]}{\hbar}\right) \ ,
\end{equation*}
with the so-called euclidean action defined as
\begin{equation*}
S[z] = \int^{\beta\hbar}_{0} \left[\frac{1}{2}m\dot{z}^2 + V (z)\right]\rmd\tau \ .
\end{equation*}

There is a crucial difference between this formula and that for the classical action of the mechanical problem: the sign in front of the potential. In fact, in the path integral formulation of quantum statistical mechanics one is led to investigate the mechanical problem defined by minus the potential. With this in mind, one may proceed along the same lines as in quantum mechanics to obtain a semiclassical approximation to the thermal density matrix element. It is given by
\begin{equation*}
\rho_{sc}(x_2, x_1) =\sum_{n=1}^{i'} \left\{\det[S^{(2)}_n]\right\}^{-1/2}\exp\left(-\frac{S_n}{\hbar}\right) \ .
\end{equation*}
Another important difference with respect to quantum mechanics is that the sum runs only over those solutions of the (euclidean) equation of motion satisfying the boundary conditions which are local minima of the euclidean action $S[z]$, whereas in quantum mechanics all solutions, i.e. any extremum, must be taken into account. That is why we use $i'$ as an upper limit of the sum, defined as the number of extrema that are minima. Clearly, $i'\le i$. In reality, we need a refinement of the focal decomposition to tell us which of the solutions are local minima.

Either in quantum mechanics or in quantum statistical mechanics, the semiclassical approximation has to sum over all, or part of, the classical paths satisfying fixed point boundary conditions. Given the pairs $(x_1, t_1)$ and $(x_2, t_2)$, or $(x_1, 0)$ and $(x_2, \beta\hbar)$, the number and type of classical trajectories are the very ingredients which lead to a focal decomposition. It should, therefore, be no surprise that the focal decomposition can be viewed as the starting point for a semiclassical calculation.

As for the renormalization procedure, it was introduced to study the behavior of classical trajectories for very short space and very long time separations of the fixed endpoints. It maps those trajectories into $n$-renormalized ones, whose time separations are shifted by $n$ half-periods, and whose space separations are scaled up to values of order one. As will be shown in the sequel, this procedure converges to an asymptotic universal family of trajectories that have a well-defined and simple functional form, and which define an asymptotic universal focal decomposition self-similar to the original one.

The natural question to pose is whether the combination of focal decomposition and renormalization can be used to calculate semiclassical expansions for propagators in the short space, long time separation of the endpoints, or analogously, for thermal density matrices for short space separation and low temperatures (long euclidean time $\beta\hbar$ is equivalent to low temperatures $T = 1/(k_B\beta)$) by using the simple asymptotic forms alluded to in the previous paragraph.

The conjecture to be investigated in a forthcoming article is that this can be done in a relatively simple way, thanks to the simple form of the asymptotes. This will bypass a much more difficult (if not impossible) calculation involving Jacobi's elliptic functions. Should our expectation be realized, we would obtain semiclassical estimates for both propagators and thermal density matrices in the short space/long time or short space/low temperature limits. Expressing those quantities in terms of energy eigenfunctions and energy eigenvalues, we have
\begin{equation*}
G(x_1, t_1; x_2, t_2) = \sum_{m=1}^\infty \psi_m^*(x_2)\psi_m(x_1)\exp(-\rmi E_m(t_2-t_1)/\hbar) \ ,
\end{equation*}
in the case of quantum mechanics, or
\begin{equation}\label{SC2}
\rho(x_2, x_1) = \sum_{m=1}^\infty \psi_m^*(x_2)\psi_m(x_1)\exp(-\beta E_m)
\end{equation}
in the case of quantum statistical mechanics.

Clearly, in the latter case, if we take $\beta$ large, only the lowest energy $E_0$ (the ground state) will contribute. Furthermore, the points $x_1$ and $x_2$ are to be taken in the limit of short space separation. If we choose one of them to be the origin, then \eqref{SC2} becomes
\begin{equation*}
\rho(x_2, x_1) \approx \rho(0, 0) \approx |\psi_0|^2\exp(-\beta E_0)  \ .
\end{equation*}
This means that the combination of focal decomposition and renormalization may lead us to a direct estimate of the ground state energy for a quantum mechanical system from the asymptotic forms obtained in this article.

\section{Main Theorems and Definitions}\label{intro_results}

We start this section by fixing notation and introducing basic definitions which will be used throughout the paper. We also state the main results to be proved in the following sections.

\subsection{Setting}
We consider mechanical systems defined by a Lagrangian function ${\mathcal L}: \R^{2} \rightarrow \R $ of the form
\begin{equation} \label{L_P}
{\mathcal L}\left(q,\frac{\rmd q}{\rmd \tau}\right)=\frac{1}{2}m\left(\frac{\rmd q}{\rmd \tau}\right)^{2}-{\mathcal V}(q) \ ,
\end{equation}
where the potential function ${\mathcal V}:\R\rightarrow \R$ is a non--isochronous potential. Furthermore, we assume that the potential ${\mathcal V}$ is a $C^{\kappa}$ map ($\kappa \ge 5$) with a Taylor expansion at a point $q^*\in\R$ given by
\begin{equation*}
{\mathcal V}(q) = {\mathcal V}(q^*) + \frac{{\mathcal V}''(q^*)}{2}(q-q^*)^2 + \frac{{\mathcal V}^{(4)}(q^*)}{4!}(q-q^*)^4 \pm O\left(\left|q-q^*\right|^{5}\right) \ ,
\end{equation*}
where ${\mathcal V}''(q^*)>0$ and ${\mathcal V}^{(4)}(q^*)\ne 0$.
The  Euler--Lagrange equation associated with \eqref{L_P} is
\begin{equation}\label{EL}
m\frac{\rmd^2 q}{\rmd \tau^2} = -\frac{\rmd {\mathcal V}}{\rmd q}(q)  \ ,
\end{equation}
and the corresponding Hamilton equations are given by
\begin{eqnarray}\label{H}
\frac{\rmd q}{\rmd \tau}&=&\frac{p}{m}\nonumber \\
\frac{\rmd p}{\rmd \tau}&=&-\frac{\rmd {\mathcal V}}{\rmd q}(q)  \ .
\end{eqnarray}
Therefore, the point $q^*$ is an elliptic equilibrium of \eqref{EL} (or equivalently, $(q^*,0)$ is an elliptic equilibrium of \eqref{H}) and thus, there is a 1--parameter family of periodic orbits covering a neighbourhood of the equilibrium point. 

\subsection{Asymptotic universal behaviour for the trajectories}
Since $q^*$ is an elliptic equilibrium of \eqref{EL} there is $\epsilon>0$ such that for all \emph{initial velocity} ${\nu\in[-\epsilon,\epsilon]}$ the solutions $q(\nu;\tau)$ of the Euler--Lagrange equation \eqref{EL} with initial conditions $q(\nu;0)=q^*$ and $\rmd q/\rmd\tau(\nu;0)=\nu$ are periodic.  Thus, the \emph{trajectories} $q:\left[-\epsilon,\epsilon\right]\times\R\rightarrow\R$ of \eqref{EL} are well-defined by $q(\nu;\tau)$ for all $\tau\in\R$ and $\nu\in[-\epsilon,\epsilon]$. Furthermore, there exist $\alpha>0$ small enough and $N \ge 1$ large enough such that, for every $n \ge N$, the \emph{$n$-renormalized trajectories} $x_{n}:\left[-1,1\right]\times[0,\alpha n]\rightarrow\R$ are well-defined by
\begin{equation*}
x_{n}(v;t)= (-1)^n~\Gamma_{n,t}^{-1}~\mu^{-1} \left[ q\left( \Gamma_{n,t} ~ \mu~ \omega~v
;\frac{n\pi-\ell t}{\omega}\right) - q^* \right]\ ,
\end{equation*}
where $\Gamma_{n,t}$ is the \emph{$(n,t)$-scaling parameter}
\begin{equation}\label{gamma}
\Gamma_{n,t} = \left(\frac{ 8 t}{3\pi n}\right)^{1/2} \ ,
\end{equation}
$\ell=\pm 1$ depending on the sign of ${\mathcal V}^{(4)}(q^*)$ and $\omega$ and $\mu$ are given by
\begin{equation}\label{omega_mu}
\omega = \left(\frac{V''(q^*)}{m}\right)^{1/2} \ , \qquad \mu = \left(\frac{3!V''(q^*)}{|V^{(4)}(q^*)|}\right)^{1/2} \ .
\end{equation}
Note that $\omega^{-1}$ and $\mu$ are the natural time and length scales for the dynamical system defined by \eqref{EL}. Furthermore, the variables $v$ and $t$ are dimensionless, as well as the $(n,t)$-scaling parameter $\Gamma_{n,t}$. Therefore, the $n$-renormalized trajectories $x_{n}(v;t)$ are dimensionless.

\begin{definition} \label{def2}
The \emph{asymptotic trajectories} $X_\ell: \left[-1,1\right]\times \R_0^+ \rightarrow \R$ are defined by
\begin{equation*}
X_\ell(v;t) = v  ~\sin \left(\ell t \left(v^2 -1 \right)  \right) \ ,
\end{equation*}
where $\ell=\pm 1$ depending on the sign of ${\mathcal V}^{(4)}(q^*)$.
\end{definition}

The rest of the paper is mainly devoted to prove the following result and some of its consequences.
\begin{theorem} \label{corolario_traj_q2_intro}
There exists $\beta>0$ small enough, such that, for every $0 < \epsilon < 1/3$, we have that
\begin{eqnarray*} 
\|x_{n}(v;t) -  X_\ell ( v;t)  \|_{C^0(\left[-1,1\right] \times [0, \beta n^{1/3-\epsilon}],\R)} & < & O\left( {n^{-3\epsilon/2 }}\right) \ ,
\end{eqnarray*}
where $\ell$ is  the sign of ${\mathcal V}^{(4)}(q^*)$.
\end{theorem}

\subsection{Asymptotic universal focal decomposition}

The asymptotic trajectories $X_\ell(v;t)$ induce an asymptotic focal decomposition of the cylinder ${\mathcal C} = \R_0^+\times \left[-1,1\right]$ 
by the sets $\sigma_i$ whose elements are pairs $(t,x)\in {\mathcal C}$ such that $X_\ell(v;t)=x$ has exactly $i$ solutions $v(t,x)\in[-1,1]$, each distinct solution corresponding to an asymptotic trajectory connecting the points $(0,0)\in {\mathcal C}$ and $(t,x)\in {\mathcal C}$. Therefore, for each $i\in\{0,1,...,\infty\}$, the set $\sigma_i\subset{\mathcal C}$ contains all points in $(t,x) \in \mathcal C$ such that there exist exactly $i$ asymptotic trajectories connecting $(0,0)\in {\mathcal C}$ and $(t,x)\in {\mathcal C}$. The following result is a consequence of theorem \ref{corolario_traj_q2_intro}.
\begin{theorem} 
There exists an asymptotic universal focal decomposition for the Euler--Lagrange equation \eqref{EL} induced by the asymptotic trajectories $X_\ell(v;t)$.
\end{theorem}

The asymptotic universal focal decomposition is shown in Figure \ref{UFD_fig}. As in the case of the focal decomposition with base point $(0, 0)$ of the pendulum equation $\ddot{x}+\sin x=0$ (see Figure \ref{UFD_pend_fig}), the asymptotic universal focal decomposition also exhibits non-empty sets $\sigma_i$ with all finite indices. However, contrary to that focal decomposition, which gives a stratification for the whole $\R^2$, our asymptotic universal focal decomposition gives only a stratification of the half cylinder ${\mathcal C} = \R_0^+\times \left[-1,1\right]$. There are two main reasons for this to happen which we pass to explain. Firstly, our renormalization scheme acts only on periodic orbits, neglecting the high-energy non-periodic orbits, which restrains $X_\ell(v;t)$ to the interval $\left[-1,1\right]$. Secondly, we have defined the renormalization operator only for positive times. Noticing that the mechanical systems we renormalize have time-reversal symmetry, one can extend the asymptotic universal focal decomposition to the cylinder $\R\times \left[-1,1\right]$ by a symmetry on the $x$ axis.

\begin{figure}[h!]
	\centering
      \psfrag{t}[cc][][0.85][0]{$t$}%
      \psfrag{x}[cc][][0.85][0]{$x$}
      \psfrag{p1}[cc][][0.85][0]{$\pi$}  
      \psfrag{p2}[cc][][0.85][0]{$2\pi$}  
      \psfrag{p3}[cc][][0.85][0]{$3\pi$}
      \psfrag{si}[cc][][0.85][0]{$\sigma_\infty$}  
      \psfrag{s0}[cc][][0.85][0]{$\sigma_0$}
      \psfrag{s1}[cc][][0.85][0]{$\sigma_1$}  
      \psfrag{s2}[cc][][0.85][0]{$\sigma_2$}
      \psfrag{s3}[cc][][0.85][0]{$\sigma_3$}
      \psfrag{s4}[cc][][0.85][0]{$\sigma_4$}  
      \psfrag{s5}[cc][][0.85][0]{$\sigma_5$}
      \psfrag{s6}[cc][][0.85][0]{$\sigma_6$}
      \psfrag{s7}[cc][][0.85][0]{$\sigma_7$}  
      \psfrag{s8}[cc][][0.85][0]{$\sigma_8$}
      \psfrag{s9}[cc][][0.85][0]{$\sigma_9$}
      \psfrag{x=1}[cc][][0.85][0]{$x=1$}
      \psfrag{x=-1}[cc][][0.85][0]{$x=-1$}
   		\includegraphics[width=85 mm,angle=-90]{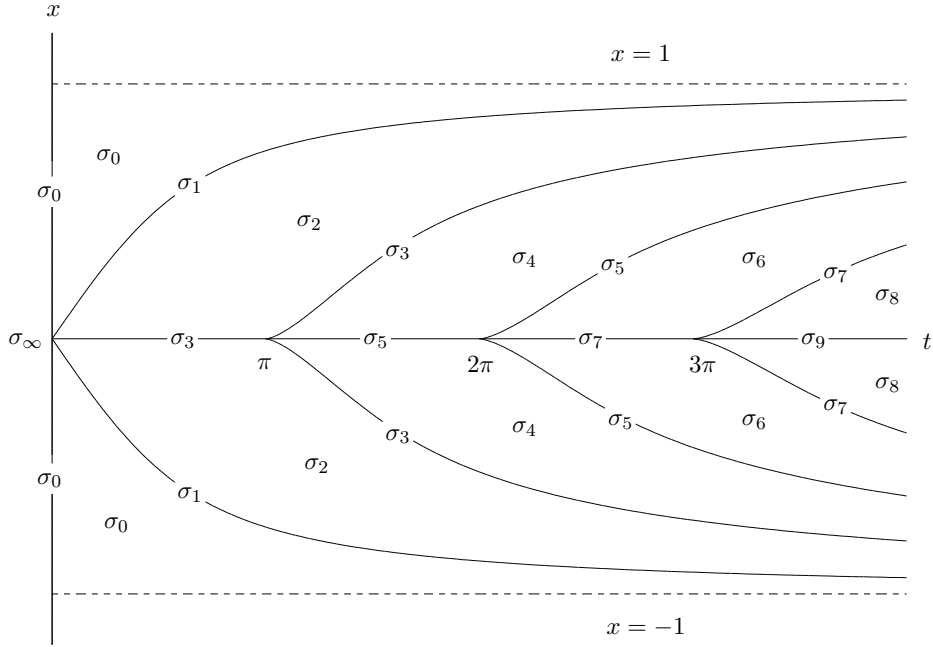}
   		\caption{The asymptotic universal focal decomposition.}
   		\label{UFD_fig}
\end{figure}

Some comments about Figure \ref{UFD_fig} are appropriate here.

For every $k\in\N$, the set $\sigma_{2k}$ is a 2-dimensional open set with two connected components. The the odd-indexed sets $\sigma_{2k-1}$ are the union of two open arcs, asymptotic to one of the lines $x=\pm 1$ and incident to the cusp-point $((k-1)\pi, 0)$, and a line segment joining the cusp points $((k-1)\pi,0)$ and $(k\pi,0)$; the lines $x=\pm 1$ are part of $\sigma_0$.

Thus, the even-indexed sets $\sigma_{2k}$ are 2-dimensional manifolds, while the odd-indexed sets $\sigma_{2k-1}$ are not manifolds because they contain the cusp-points $(k\pi,0)$.

On the other hand, if we decompose the odd-indexed sets $\sigma_{2k-1}$ into a cusp-point plus three 1-dimensional manifolds (two open arcs and one line segment), then we get a decomposition of the whole plane into a collection of disjoint connected manifolds. The above decomposition of ${\mathcal C}$ is an example of what is called a stratification of ${\mathcal C}$, the strata being the disjoint connected manifolds into which ${\mathcal C}$ was decomposed. Hence $\sigma_1$ consists of two 1-dimensional strata, $\sigma_2$ consists of two 2-dimensional strata, $\sigma_3$ consists of three 1-dimensional strata and one 0-dimensional strata and so on. To complete the picture, $\sigma_0$ consists of two 2-dimensional strata plus the $x$-axis minus the origin which belongs to $\sigma_\infty$.

\section{The renormalization procedure}\label{Ren_sect}

In this section we introduce an affine change of coordinates of space and time that enables us to map trajectories of the Lagrangian system \eqref{L_P} to trajectories of a dimensionless Lagrangian system, simplifying the study of the asymptotic properties of such trajectories.

\subsection{An affine change of coordinates and the perturbed quartic potentials}
In order to simplify the proofs throughout the paper, and without loss of generality, we apply a change of coordinates of space and time to the Lagrangian \eqref{L_P} and corresponding Euler-Lagrange equation \eqref{EL}. The new dimensionless coordinates $(x,t)$ are defined by
\begin{equation}\label{CC_P}
x = \mu^{-1}\left(q-q^*\right) \ , \qquad t=\omega\tau \ ,
\end{equation}
where $\omega$ and $\mu$ are as given in \eqref{omega_mu}.
We obtain a normalized Lagrangian function
\begin{equation}\label{NL_P}
L(x,\overset{.}{x})=\frac{1}{2}\overset{.}{x}^{2}-V(x)  \ ,
\end{equation}
with potential function $V(x)$ of the form
\begin{equation}\label{NV_P}
V(x)= V(0) + \frac{1}{2}x^2 + \frac{\ell}{4}x^4 + f(x)  \ ,
\end{equation}
where
\begin{itemize}
\item[a)] $f(x) \in O\left(|x|^5\right)$ is a $C^{\kappa}$ map ($\kappa \ge 5$);
\item[b)] $\ell = 1$ if ${\mathcal V}^{(4)}(q^*)>0$ and $\ell = -1$ if ${\mathcal V}^{(4)}(q^*)<0$.
\end{itemize}
Since adding a constant to the potential does not change the form of the associated Euler--Lagrange equation, for simplicity of notation and without loss of generality we rescale the total energy associated with the Lagrangian system \eqref{NL_P} so that that $V(0)=0$. In the next definition, we introduce the quartic potentials and the perturbed quartic potentials

\begin{definition}
The \emph{quartic potentials} $V_{\ell}:\R \to \R$ are defined by
\begin{equation}\label{NPV}
V_{\ell}(x)= \frac{1}{2}x^2 + \frac{\ell}{4}x^4 \ , \qquad \ell \in \{-1,1\} \ . 
\end{equation}
The $C^5$ \emph{perturbed quartic potentials} $V:\R \to \R$ are defined by
\begin{equation}\label{NV}
V(x)= \frac{1}{2}x^2 + \frac{\ell}{4}x^4 + f(x)  \ ,
\end{equation}
where $\ell\in \{-1,1\}$ and $|f^{(i)}(x)| \in O\left(|x|^{5-i}\right)$ for $i\in\{0,\ldots,5\}$.
\end{definition}

\begin{figure}[h!]
	\centering
			\psfrag{0}[cc][][0.85][0]{$0$}
      \psfrag{1}[ct][][0.85][0]{$1$}
      \psfrag{-1}[ct][][0.85][0]{$-1$}  
      \psfrag{x}[cc][][0.85][0]{$x$}    
      \psfrag{V1x}[lc][][0.85][0]{$V_1(x)$}
      \psfrag{V-1x}[lc][][0.85][0]{$V_{-1}(x)$}
      \psfrag{0.2}[rc][][0.85][0]{$0.2$}
   		\subfloat[][]{\includegraphics[width=0.31\textwidth,angle=-90]{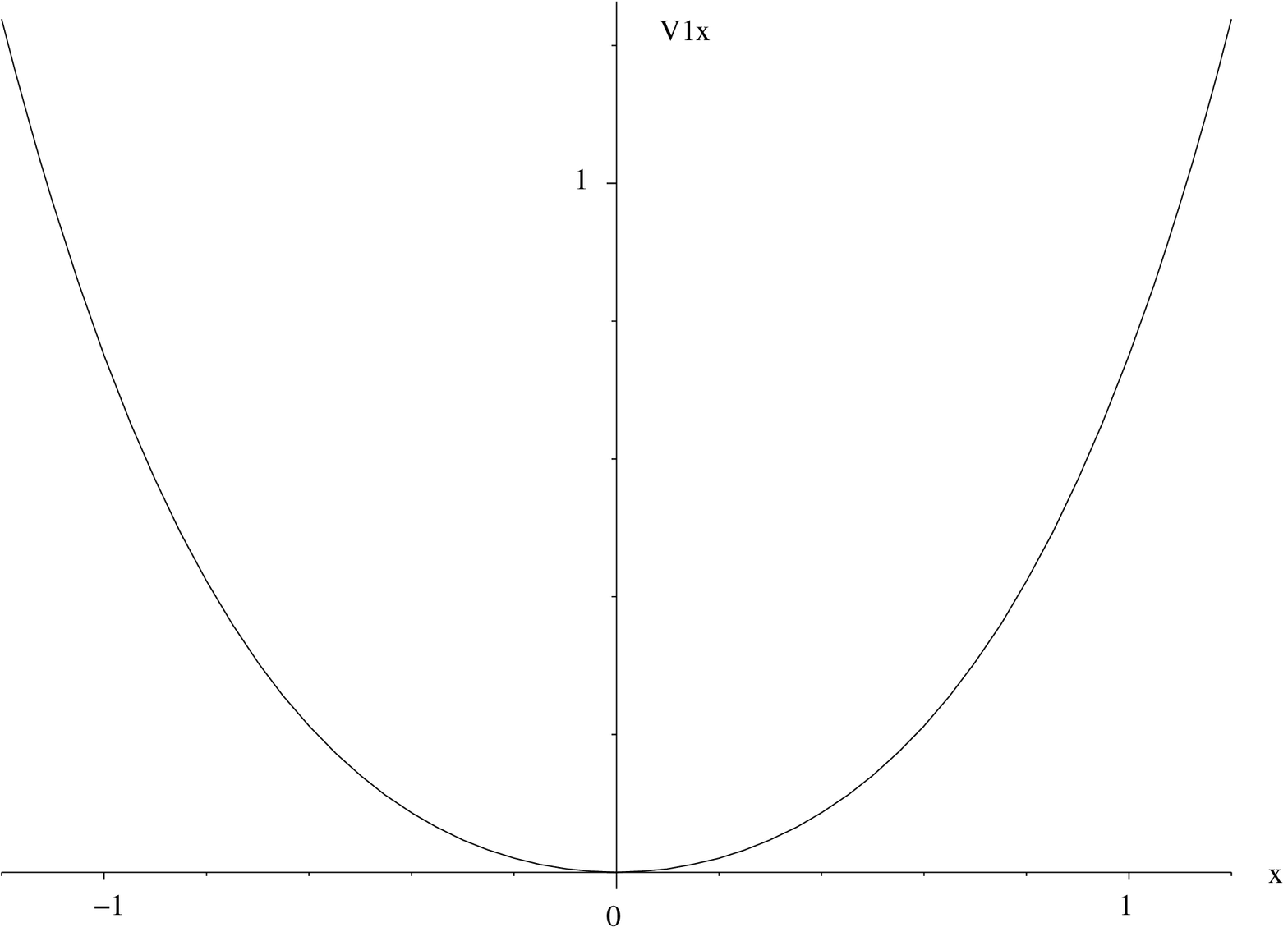}\label{V_pos_img}}\qquad \qquad
   		\subfloat[][]{\includegraphics[width=0.31\textwidth,angle=-90]{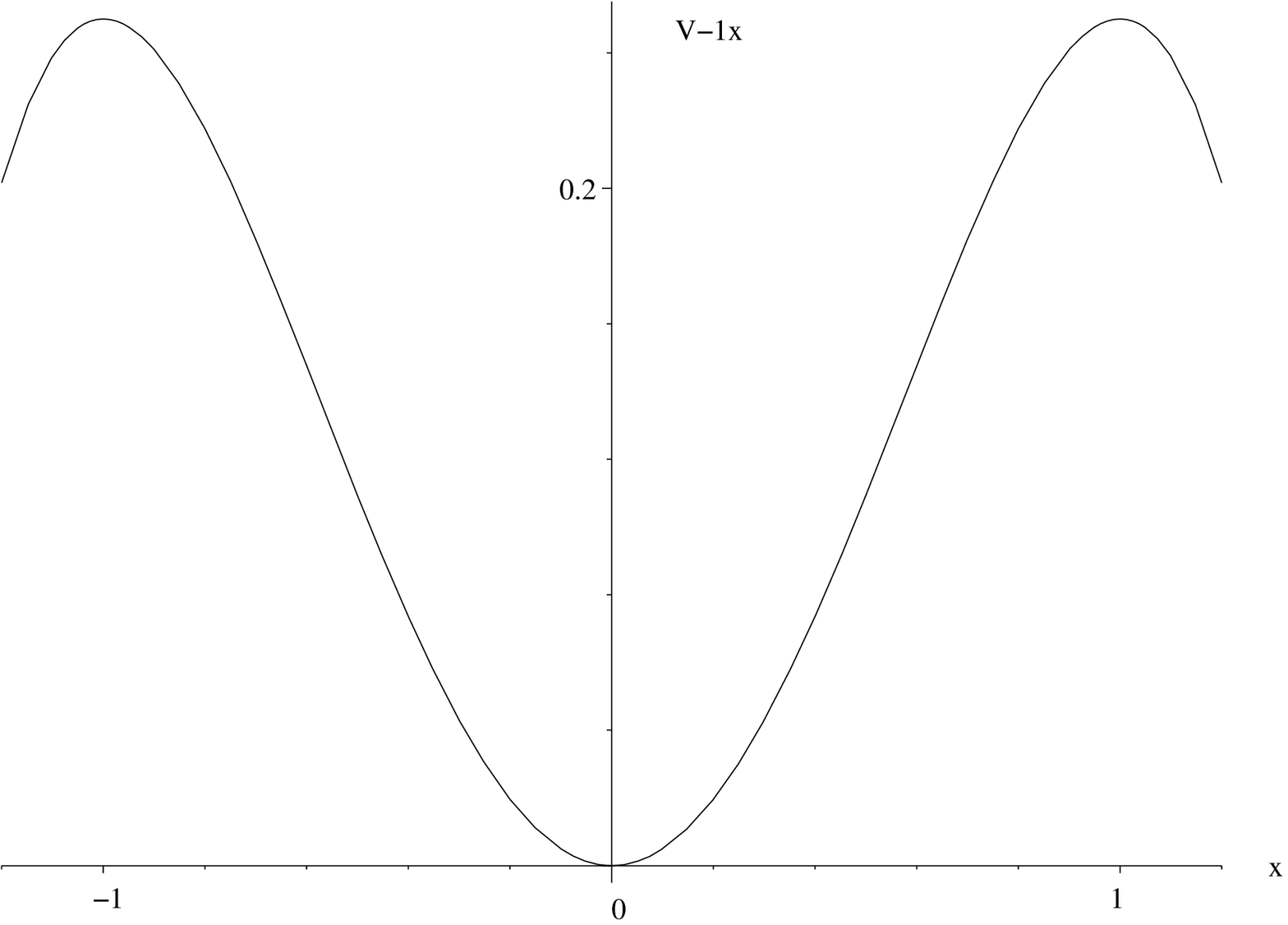}\label{V_neg_img}}
\caption{The quartic potentials $V_\ell(x)$.}
\end{figure}

In the new coordinates $(x,t)$, the Euler-Lagrange equation associated with the normalized Lagrangian \eqref{NL_P} reduces to
\begin{equation}\label{NEL_P}
\overset{..}x = -(x + \ell x^3 + f'(x))  \ ,
\end{equation}
where the dot denotes differentiation with respect to the variable $t$.
The next Lemma describes the relation between the solutions associated with the Lagrangian \eqref{L_P} and the normalized Lagrangian \eqref{NL_P}. The result follows easily from the change of coordinates \eqref{CC_P}. 
\begin{lemma} \label{rel_qx}
The solution $q(\nu;\tau)$ of Euler-Lagrange equation \eqref{EL} with initial conditions $q(\nu;0)=q^*$ and
$\frac{\rmd q}{\rmd \tau}(\nu;0)=\nu$ is related to the solution $x(v;t)$ of the Euler--Lagrange equation
\eqref{NEL_P} with initial conditions $x(v;0)=0$ and $\overset{.}{x}(v;0)=v$ by
\begin{equation*}
q(\nu;\tau) =\mu x\left(\mu^{-1}\omega^{-1}\nu;\omega\tau\right) + q^* \ ,
\end{equation*}
where $\omega$ and $\mu$ are as given in \eqref{omega_mu}.
\end{lemma}

\subsection{The renormalization procedure}

Since $x=0$ is a local minimum of the potential $V(x)$ given by \eqref{NV_P} there is $v_{Max}>0$ such that for all $|v|<v_{Max}$ the solutions $x(v;t)$ of the Euler--Lagrange equation \eqref{NEL_P} with initial conditions $x(v;0)=0$ and $\dot{x}(v;0)=v$ are periodic. Furthermore, there exist $\alpha>0$ small enough and $N \ge 1$ large enough such that, for every $n \ge N$, the \emph{$n$-renormalized trajectories} $x_{n}:\left[-1,1\right]\times[0,\alpha n]\rightarrow\R$ are well-defined by 
\begin{equation*}
x_{n}(v;t)=(-1)^{n}~\Gamma_{n,t}^{-1}~x\left ( \Gamma_{n,t}~v;n\pi -\ell t\right ) \ ,
\end{equation*}
where $\Gamma_{n,t}$ is the $(n,t)$-scaling parameter defined in \eqref{gamma}. Using Lemma \ref{rel_qx}, we obtain that Theorem \ref{corolario_traj_q2_intro} follows as a consequence of the following result.

\begin{proposition}\label{main}
There exist $K>0$ and $N \ge 1$ such that for every $n \ge N$ and every $t\in\left[0,Kn\right]$, the $n$-renormalized trajectories ${x}_{n}(v;t)$ converge to the asymptotic trajectories $X_\ell(v;t)$. Furthermore, the
following bound is satisfied
\begin{equation*}
\left |{x}_{n}(v;t)-X_{\ell}\left( v;t\right)\right | < O\left(\frac{|t|}{n}; \frac{|t|^{2}}{n}; \frac{|t|^{3/2}}{n^{1/2}}\right)  \\
\end{equation*}
for all $v \in [-1,1]$.
\end{proposition}

The rest of this paper is devoted to prepare the proof of Proposition \ref{main}, which will be completed in section \ref{trajectories}.


\section{Quartic potentials} \label{QP}
In this section we will make use of Jacobian elliptic functions (see \cite{Abramowitz_Stegun,Bowman,Whittaker_Watson}) to compute explicit solutions, with suitable initial conditions, for trajectories of a given class of Lagrangian systems. Furthermore, we will compute the first terms of the Taylor expansion of such solutions in terms of their initial velocities. In the next section, we will use these Taylor expansions to study the trajectories associated to the more general potentials described in section \ref{Ren_sect}.

We will consider Lagrangian functions of the form
\begin{equation*}
L_\ell(x,\dot{x})=\frac{1}{2}\dot{x}^2 - V_\ell(x) \ , 
\end{equation*}
where $V_\ell(x)$ is a quartic potential introduced in \eqref{NPV}, and compute explicit solutions of the corresponding Euler--Lagrange equation
\begin{eqnarray}\label{NPEL}
&&\overset{..}{x}= - x - \ell x^3 \nonumber\\
&&x(0)=0 \\
&&\dot{x}(0)=v \nonumber \ .        
\end{eqnarray}
We denote such solutions by $x_\ell(v;t)$, where $\ell\in\{-1,1\}$.

Recall that $x=0$ is a local minimum of $V_\ell(x)$. Thus, there exists $v_{Max}>0$ such that for all initial velocities $v$ satisfying $|v|<v_{Max}$, the solutions of \eqref{NPEL} are periodic.

Let $\ell\in\{-1,1\}$ and let $v>0$ be small enough so that the solutions of \eqref{NPEL} are periodic. We define the \emph{$(\ell,v)$--coefficients} $a_\ell$ and $b_\ell$ by
\begin{eqnarray}\label{a_b}
a_\ell:=a_\ell(v) &=& \left(1+\left(1+2\ell v^2\right)^{1/2}\right)^{1/2} \nonumber \\
b_\ell:=b_\ell(v) &=& \left(-\ell+\ell\left(1+2\ell v^2\right)^{1/2}\right)^{1/2} \ .
\end{eqnarray}

\begin{lemma}\label{exact-solutions_1} 
Let $\ell=1$. The solution $x_1(v;t)$ of \eqref{NPEL} is given by
\begin{equation*}
x_1(v;t) = A_1\sd\left(\lambda_1 t ; m_1\right) \ ,
\end{equation*}
where:
\begin{itemize}
\item[(i)] $\sd(t ; m)$ is a Jacobian elliptic function.
\item[(ii)] the amplitude $A_1$, the frequency $\lambda_1$, and the parameter $m_1$ are given by
\begin{eqnarray}\label{coef_1}
A_1:=A_1(v)&=& \sign(v)a_1b_1/\left({a_1}^2+{b_1}^2\right)^{1/2} \nonumber \\
\lambda_1:=\lambda_1(v) &=& \left(({a_1}^2+{b_1}^2)/2\right)^{1/2} \\
m_1:=m_1(v) &=& {b_1}^2/\left({a_1}^2+{b_1}^2\right) \nonumber \ .
\end{eqnarray}
\item[(iii)] $a_1$ and $b_1$ are as given in \eqref{a_b}.
\end{itemize}
\end{lemma}
\begin{proof}
Noting that the solution of \eqref{NPEL} is also the solution of the first order differential equation 
\begin{equation*}
\frac{1}{2}\dot{x}^2 + V_1(x)= \frac{1}{2}v^2 
\end{equation*}
with initial condition $x_1(0)=0$, integrating the above equation we obtain that the solution $x_1(v;t)$ is implicitly determined by 
\begin{equation*}
\int_{0}^x \left(\frac{1}{2}v^2- V_1(z)\right)^{-1/2} \rmd z= \sign(v) t/\sqrt{2} \ .
\end{equation*}
Using formula 17.4.51 from \cite{Abramowitz_Stegun}, we get
\begin{equation*}
\frac{1}{\lambda_1}\left({\sd}^{-1}\left(x_1/A_1 ; m_1\right)\right)= \sign(v) t/\sqrt{2} \ ,
\end{equation*}
where $A_1$, $\lambda_1$ and $m_1$ are defined in \eqref{coef_1} in terms of $a_1$ and $b_1$ given in \eqref{a_b}. The lemma follows by solving the above equality with respect to $x_1$ to obtain $x_1(v;t)$.
\end{proof}

The proof of the next lemma is analogous to the proof of lemma \ref{exact-solutions_1} (using formula 17.4.45 from \cite{Abramowitz_Stegun}). We skip its proof.
\begin{lemma}\label{exact-solutions_m1} 
Let $\ell=-1$. The solution $x_{-1}(v;t)$ of \eqref{NPEL} with $|v|<v_{Max}$ is given by
\begin{equation*}
x_{-1}(v;t) = A_{-1}\sn\left(\lambda_{-1} t ; m_{-1}\right) \ ,
\end{equation*}
where:
\begin{itemize}
\item[(i)] $\sn(t ; m)$ is the Jacobian elliptic sinus.
\item[(ii)] the amplitude $A_{-1}$, the frequency $\lambda_{-1}$ and the parameter $m_{-1}$ are given by
\begin{eqnarray*}
A_{-1}:=A_{-1}(v) &=& \sign(v) b_{-1} \nonumber \\
\lambda_{-1}:=\lambda_{-1}(v) &=& a_{-1}/\sqrt{2} \\
m_{-1}:=m_{-1}(v) &=& (b_{-1}/a_{-1})^2 \nonumber \ .
\end{eqnarray*}
\item[(iii)] $a_{-1}$ and $b_{-1}$ are as given in \eqref{a_b}.
\end{itemize}
\end{lemma}

In the two next lemmas we provide estimates for the amplitude $A_{\ell}$, frequency $\lambda_{\ell}$ and parameter $m_{\ell}$. The estimates are obtained by finding the values of the successive derivatives of $A_{\ell}$, $\lambda_{\ell}$ and $m_{\ell}$ when $v=0$.
\begin{lemma}\label{lemma_est_1}
Let $\ell=1$. There exists $v_{Max}>0$ such that for all $v\in(-v_{Max},v_{max})$, we have that
\begin{align*}
A_1 & \in  v - \frac{1}{2}v^{3} \pm O\left(|v|^5\right) \\
\lambda_1 & \in 1 + \frac{1}{2}v^{2} \pm O\left(v^4\right) \\
m_1& \in \frac{1}{2}v^{2}  \pm O\left(v^4\right) \ .
\end{align*}
\end{lemma}

\begin{lemma}\label{lemma_est_m1}
Let $\ell=-1$. There exists $v_{Max}>0$ such that for all $v\in(-v_{Max},v_{Max})$, we have that
\begin{align*}
A_{-1} & \in  v + \frac{1}{4}v^{3} \pm O\left(|v|^5\right) \\
\lambda_{-1} & \in 1 - \frac{1}{4} v^2 \pm O\left(v^4\right) \\
m_{-1} & \in \frac{1}{2}v^{2}  \pm O\left(v^4\right) \ .
\end{align*}
\end{lemma}

\begin{lemma} \label{Trig_sum}
Let $j_1=\frac{1}{2}$ and $j_{-1}=- \frac{1}{4}$ and define $e_\ell(v) = 1 + j_\ell v^2$, so that we have
\begin{equation*}
\lambda_{\ell} = e_\ell(v) \pm O(v^4) \ .
\end{equation*}
There is $v_0>0$ such that for all $|v|<v_0$, the following estimates hold
\begin{align*}
\sin(\lambda_\ell(v) t) &\in \sin \left(e_\ell(v) t \right) \pm O(t v^4) \\ 
\sin \left(e_\ell(v) t \right) & \in \sin(t) \pm O(t v^2)  \ . 
\end{align*}
\end{lemma}

\begin{proof}
We use the estimates for $\lambda_\ell$ in Lemmas \ref{lemma_est_1} and \ref{lemma_est_m1}, the definition of $e_\ell$ and the trigonometry formulas for the sine and cosine of the sum of two angles. 
\end{proof}

The next Lemma gives us two classical estimates on Jacobian elliptic functions (formulas 16.13.1 and 16.13.3 of \cite{Abramowitz_Stegun}).
\begin{lemma}\label{aux2}
There is $m_0>0$ small enough such that for all $|m|<m_0$, the following estimates hold
\begin{align*}
\sn(t;m) & \in \sin(t) - \frac{1}{4}m\left(t-\sin(t)\cos(t)\right)\cos(t) \pm O\left(m^2\right)  \\
\dn(t;m) & \in 1 - \frac{1}{2}m \sin^2(t) \pm O\left(m^2\right)  \ .
\end{align*} 
\end{lemma}

In the next Lemma we rewrite the classical estimates of Lemma \ref{aux2} in a form that will be useful for the proof of Lemma \ref{np-lemma3}.
\begin{lemma} \label{est_JEF}
Let $g_{\ell}(t)$ be given by
\begin{equation*}
g_{\ell}(t)= t\cos(t)+\ell \left(\sin(t)\cos^2(t)+2\sin(t)\right) \ .
\end{equation*}
There is $m_0>0$ such that for all $|m|<m_0$, the following estimates hold 
\begin{align}
\sd(t;m) & \in \sin( t) - \frac{1}{4}m\left(g_1\left( t \right)-4\sin\left( t\right)\right) \pm O\left(m^2\right) \label{app1} \\
\sn\left( t;m\right) & \in \sin\left( t\right) - \frac{1}{4}m\left(g_{-1}\left( t\right) + 2\sin\left( t\right)\right) \pm O\left(m^2\right)  \label{app2} \ .
\end{align}
\end{lemma}

\begin{proof}
For the first inequality, using the definition of the Jacobian elliptic function $\sd(t;m)$ (see \cite{Abramowitz_Stegun}), we obtain
\begin{equation}
\sd(t;m)= \frac{\sn(t;m)}{\dn( t;m)} \label{sd_def} \ .
\end{equation}
Putting together \eqref{sd_def} and Lemma \ref{aux2}, we get that there is $m_0>0$ such that for all $|m|<m_0$, we get
\begin{eqnarray}
\sd( t;m)&=& \frac{\sn(t;m)}{\dn( t;m)} \nonumber \\
&\in & \frac{\sin( t) - m\left( t-\sin( t)\cos( t)\right)\cos( t)/4 \pm O\left(m^2\right)}{1-m\sin^2( t)/2\pm O\left(m^2\right)} \label{Japp1.1} \ .
\end{eqnarray}
By Taylor's series, we get
\begin{equation}
\left(1-\frac{1}{2}m\sin^2(t)\pm O\left(m^2\right)\right)^{-1} \subset 1+\frac{1}{2}m\sin^2( t)\pm O\left(m^2\right)\label{Japp1.2} \ .
\end{equation}
Putting together inclusions \eqref{Japp1.1} and \eqref{Japp1.2} and the definition of $g_1(t)$, we get
\begin{equation*}
\sd( t;m) \in \sin( t) - \frac{1}{4}m\left(g_1\left( t \right)-4\sin\left( t\right)\right) \pm O\left(m^2\right)  \ ,
\end{equation*}
which proves \eqref{app1}.

The second equality follows from Lemma \ref{aux2} and the definition of $g_{-1}\left(t\right)$. We obtain that there is $m_0>0$ such that for all $|m|<m_0$, we get
\begin{equation*} 
\sn\left( t;m\right) \in \sin\left( t\right) - \frac{1}{4}m\left(g_{-1}\left( t\right) + 2\sin\left( t\right)\right) \pm O\left(m^2\right) \ ,
\end{equation*}
as required.
\end{proof}

\begin{lemma}\label{np-lemma3}
 There is $v_0>0$ such that for every $|v|<v_0$, the solution $x_\ell(v;t)$ of \eqref{NPEL} satisfies the following estimate
\begin{equation}
x_\ell \left(v;t\right) \in v\sin (e_\ell(v) t) - \frac{1}{8}v^3 g_\ell(e_\ell(v) t) \pm O(t |v|^5) \pm O\left(|v|^{5}\right) \nonumber \ ,
\end{equation}
where $e_\ell$ is defined in Lemma \ref{Trig_sum} and $g_{\ell}$ is defined in Lemma \ref{est_JEF}.

In particular, for every $\ell \in \{-1,1\}$, we get 
\begin{equation}
x_\ell\left(v;t\right)\in v\sin \left( t\right)\pm O(t |v|^3) \pm O\left(|v|^{3}\right) \nonumber ) \ .
\end{equation}
\end{lemma}

\begin{proof} We separate the proof into two cases: $\ell=1$ and $\ell=-1$.

{\it Case $\ell=1$}. For simplicity of notation, we will use $A=A_1$, $\lambda=\lambda_1$, $m=m_1$, $g=g_1$ and $e=e_1$.

From Lemma \ref{exact-solutions_1}, we have that
\begin{equation} \label{s1}
x_1(v;t) = A \sd\left(\lambda t ; m\right)\ .
\end{equation}
Combining \eqref{app1} in Lemma \ref{est_JEF} with \eqref{s1}, we obtain
\begin{eqnarray} \label{s1.0}
x_1\left(v;t\right) &\in & A \left( \sin(\lambda t) - \frac{1}{4}m\left(g\left(\lambda t \right)-4\sin\left(\lambda t\right)\right) \pm O\left({m}^2\right) \right)
\end{eqnarray}
Combining the estimate for $m$ in Lemma \ref{lemma_est_1} with \eqref{s1.0}, we obtain
\begin{eqnarray} \label{s1.1}
x_1\left(v;t\right)&\in & A\left(\sin(\lambda t) - \frac{1}{4}\left(\frac{1}{2}v^{2}  \pm O\left(v^4\right)\right)\left(g\left(\lambda t \right)-4\sin\left(\lambda t\right)\right) \pm O(v^4)\right) \nonumber  \\
&\subset & A\left(\sin(\lambda t) - \frac{1}{8}v^2\left(g\left(\lambda t \right)-4\sin\left(\lambda t\right)\right) \pm O(v^4)\right)  \ .
\end{eqnarray}
Putting together the estimate for $A$ in Lemma \ref{lemma_est_1} and \eqref{s1.1}, we get
\begin{eqnarray}
x_1\left(v;t\right)&\in &\left(v-\frac{1}{2}v^3 \pm O\left(|v|^5\right)\right)\left(\sin(\lambda t) - \frac{1}{8}v^2\left(g\left(\lambda t \right)-4\sin\left(\lambda t\right)\right) \pm O(v^4)\right) \nonumber  \\
&\subset &v\sin (\lambda t) - \frac{1}{8}v^3 g\left(\lambda t \right) \pm O\left(|v|^{5}\right) \label{s1.2}\ .
\end{eqnarray}
Applying Lemma \ref{Trig_sum} to \eqref{s1.2}, we get
\begin{eqnarray}
x_1\left(v;t\right)&\in &v\sin (e t) - \frac{1}{8}v^3 g\left(e t \right) \pm O\left(t |v|^{5}\right) \pm O\left( |v|^{5}\right) \nonumber \\
&\subset &v\sin (e t) - \frac{1}{8}v^3 g\left( t \right)\pm O(t |v|^5) \pm O\left(|v|^{5}\right) \label{x-estp1} \ ,
\end{eqnarray}
as required.

{\it Case $\ell=-1$}. For simplicity of notation, we will use $A=A_{-1}$, $\lambda=\lambda_{-1}$, $m=m_{-1}$, $g=g_{-1}$ and $e=e_{-1}$.

From Lemma \ref{exact-solutions_m1}, we have that
\begin{equation} \label{s2}
x_{-1}(v;t) = A\sn\left(\lambda t ; m\right) \ .
\end{equation}
Combining \eqref{app2} in Lemma \ref{est_JEF} with \eqref{s2}, we get
\begin{equation} \label{s2.0}
x_{-1}(v;t) = A\left( \sin\left(\lambda t\right) - \frac{1}{4}m\left(g\left(\lambda t\right) + 2\sin\left(\lambda t\right)\right) \pm O\left({m}^2\right) \right) \ .
\end{equation}
Combining the estimate for $m$ in Lemma \ref{lemma_est_m1} and \eqref{s2.0}, we obtain
\begin{eqnarray} \label{s2.1}
x_{-1}\left(v;t\right)&\in & A \left(\sin\left(\lambda t\right) - \frac{1}{4}\left(\frac{1}{2}v^{2}  \pm O\left(v^4\right)\right)\left(g\left(\lambda t\right) + 2\sin\left(\lambda t\right)\right) \pm O(v^4) \right) \nonumber  \\
&\subset & A \left(\sin\left(\lambda t\right) - \frac{1}{8}v^2\left(g\left(\lambda t\right) + 2\sin\left(\lambda t\right)\right) \pm O(v^4) \right)  \ .
\end{eqnarray}
Putting together the estimate for $A$ in Lemma \ref{lemma_est_m1} and \eqref{s2.1}, we get
\begin{eqnarray}
x_{-1}\left(v;t\right)&\in & \left(v + \frac{1}{4}v^3 \pm O\left(|v|^5\right)\right) \left(\sin(\lambda t) - \frac{1}{8}v^2\left(g\left(\lambda t\right) + 2\sin\left(\lambda t\right)\right) \pm O(v^4) \right) \nonumber  \\
&\subset & v\sin(\lambda t) - \frac{1}{8}v^3 g\left(\lambda t\right) \pm O\left(|v|^5\right) \ . \label{s2.2}
\end{eqnarray}
Using Lemma \ref{Trig_sum} in \eqref{s2.2}, we obtain that
\begin{eqnarray}
x_{-1}\left(v;t\right)&\in &v\sin(e t) - \frac{1}{8}v^3 g\left(e t\right) \pm O(t |v|^5)\pm O\left(|v|^{5}\right) \nonumber \\
&\subset &v\sin (e t) - \frac{1}{8}v^3g\left(t\right)\pm O(t |v|^5) \pm O\left(|v|^{5}\right) \label{x-estp2} \ ,
\end{eqnarray}
as required.

Finally, applying again Lemma \ref{Trig_sum} to inequalities  \eqref{x-estp1} and \eqref{x-estp2},  for $\ell \in \{-1,1\}$, we get
\begin{equation*}
x_\ell\left(v;t\right) \in  v\sin (t)\pm O(t |v|^3) \pm O\left(|v|^{3}\right) \ .
\end{equation*}
\end{proof}


\section{Perturbed quartic potentials $V$}  \label{Perturbed_Potentials}

In this section, we will consider  \emph{perturbed quartic potentials} $V(x)$ that are perturbations of the quartic potentials $V_\ell(x)$
near the elliptic fixed point $(0,0)$ of the first order differential system  associated to the Euler--Lagrange equation \eqref{NEL_P}.
We will prove estimates for the solutions $x(v;t)$ of the Euler--Lagrange equation \eqref{NEL_P} associated to
$V(x)$ using the solutions $x_\ell(v;t)$ associated to the quartic potentials $V_\ell(x)$.

Along this paper,  we will denote by $x(v;t)$ the solutions of the
Euler--Lagrange equation \eqref{NEL_P} with initial conditions $x(v;0)=0$ and $\overset{.}x(v;0)=v$.

Let $C^2_0\left([a,b],\R\right)$ be the set of all maps $\psi \in C^2\left([a,b],\R\right)$ such that
$\psi(0)=0$. Let  $U:C^2_0\left([0,T],\R\right)  \rightarrow C^0\left([0,T],\R\right) \times \R $ be the linear operator given by
\begin{equation*}
U \psi = \left(\overset{..}{\psi} + \psi , \overset{.}{\psi}(0) \right) \ .
\end{equation*}

\begin{lemma} \label{lin_op_L}
The linear operator $U:C^2_0\left([0,T],\R\right)  \rightarrow C^0\left([0,T],\R\right)$
is invertible and its inverse $U^{-1}: C^0\left([0,T],\R\right) \times \R \rightarrow
C^2_0
\left([0,T],\R\right)$ is a linear operator with bounded norm.
\end{lemma}

\begin{proof}
The existence and uniqueness of $U^{-1}\left(f(t),v\right)$ is  guaranteed by the theorem
of existence and uniqueness of solutions of differential equations. For every  $f\in C^0\left([0,T],\R\right)$
and every $v \in \R$, let us show that
\begin{equation*} 
\psi_v(t) = v\sin(t) - \left(\int_0^t {f(s)\sin(s)\rmd s}\right)\cos(t) + \left(\int_0^t
{f(s)\cos(s)\rmd s}\right)\sin(t)
\end{equation*}
is equal to $U^{-1}(f,v)$. The first derivative of $\psi_v(t)$ is given by
$$
\overset{.} \psi_v(t)=v\cos(t) + \left(\int_0^t {f(s)\sin(s)\rmd s}\right)\sin(t) + \left(\int_0^t
{f(s)\cos(s)\rmd s}\right)\cos(t) \ .
$$
The second derivative of $\psi_v(t)$ is given by
\begin{equation*}
\overset{..}\psi_v(t) =-v\sin(t) + f(t) + \left(\int_0^t {f(s)\sin(s)\rmd s}\right)\cos(t) - \left(\int_0^t
{f(s)\cos(s)\rmd s}\right)\sin(t) \ .
\end{equation*}
Hence, $\overset{..}\psi_v(t)+\psi_v(t)=f(t)$, $\psi_v(t)=0$ and $\overset{.} \psi_v(0)=v$ which implies that
$U(\psi_v(t))= \left(f(t),v\right)$. Therefore, the inverse  $U^{-1}: C^0\left([0,T],\R\right) \times \R
\rightarrow C^2_0\left([0,T],\R\right)$  of $U$ is a linear operator, and
\begin{equation*}
\left\|U^{-1}\left(f(t),v\right)\right\|_{C^2\left([0,T],\R\right)} \le \left(1+2T\right)\left\|f\right\|_{C^0\left([0,T],\R\right)} + |v| \ .
\end{equation*}
Therefore, the linear operator $U^{-1}$ has bounded norm. 
\end{proof}

\begin{theorem} \label{Imp_arg}
For every $T>0$, there exist     $\delta_1 > 0$, $\delta_2 > 0$, $\delta_3 > 0$ and $K>0$ with the following
properties: For every $\left|v\right|<\delta_1$ and every $\epsilon \in
C^2_0\left(\left[-\delta_2,\delta_2\right],\R\right)$, with
$\left\|\epsilon\right\|_{C^2\left(\left[-\delta_2,\delta_2\right],\R\right)}<\delta_3$, the ordinary differential equation
\begin{equation*}
\overset{..}{x} + x + \ell x^3 + \epsilon(x)=0
\end{equation*}
has a unique solution $x_{\epsilon}(v;t) \in C^2\left(\left[0,T\right],\R\right)$ with $x_{\epsilon}(v;0)=0$
and $\overset{.}x_{\epsilon}(v;0)=v$, where $\ell= \pm 1$. Furthermore, we have that
\begin{equation} \label{est_der_x}
\left\|{\frac{\delta x_{\epsilon}}{\delta \epsilon}}_{(x_{\epsilon}(v;t),\epsilon(x_{\epsilon}(v;t)))}
 \epsilon\right\|_{C^2\left(\left[0,T\right],\R\right)} <
K\left\|\epsilon(x_{\epsilon}(v;t))\right\|_{C^0\left(\left[0,T\right],\R\right)} \ .
\end{equation}
\end{theorem}

\begin{proof}
Let $\Theta: C^2_0\left([0,T],\R\right) \times
C^2_0\left(\left[-\delta,\delta\right],\R\right) \rightarrow C^0\left([0,T],\R\right) \times \R $, be given by
\begin{equation*}
\Theta\left(x(t), \epsilon(x)\right)= \left(\overset{..}{x} + x + \ell x^3 + \epsilon(x)
,\overset{.}{x}(0)\right)  \ .
\end{equation*}
The nonlinear operator $\Theta$ is $C^1$, with partial derivative ${\frac{\delta \Theta}{\delta
x}}_{\left(x(t),\epsilon(x)\right)}: C^2_0\left([0,T],\R\right)\rightarrow C^0\left([0,T],\R\right) \times \R
$ given by
\begin{equation*} 
{\frac{\delta \Theta}{\delta x}}_{\left(x(t),\epsilon(x)\right)} \psi =\left(\overset{..}{\psi} + \psi + 3\ell
x^2\psi + D\epsilon(x) \psi, \overset{.}{\psi}(0) \right) \ ,
\end{equation*}
and with partial derivative ${\frac{\delta \Theta}{\delta \epsilon}}_{\left(x(t),\epsilon(x)\right)}:
C^2_0\left(\left[-\delta,\delta\right],\R\right) \rightarrow C^0\left([0,T],\R\right) \times \R$ given by
\begin{equation*} 
{\frac{\delta \Theta}{\delta \epsilon}}_{\left(x(t),\epsilon(x)\right)} \alpha = \left(\alpha,0\right) \ .
\end{equation*}
By Lemma \ref{lin_op_L}, the linear operator $U:C^2_0\left([0,T],\R\right) \rightarrow
C^0\left([0,T],\R\right) \times \R $ given by
\begin{equation*}
U \psi = \left(\overset{..}{\psi} + \psi , \overset{.}{\psi}(0) \right)
\end{equation*}
is invertible and its inverse has bounded norm. Furthermore, we have that
\begin{eqnarray*}
\left\|U\psi - {\frac{\delta \Theta}{\delta x}}_{\left(x(t),\epsilon(x)\right)} \psi
\right\|_{C^0\left([0,T],\R\right) \times \R}  & \le & \left\|\left(3\ell x^2 \psi +
D\epsilon(x)\psi,0\right)\right\|_{C^0\left([0,T],\R\right) \times \R}
\nonumber \\
& \le & \left(3 C {\delta_2}^2 + \delta_3\right)\left\|\psi\right\|_{C^2\left([0,T],\R\right)} \ .
\end{eqnarray*}
Hence, there is $\delta_2>0$ and $\delta_3>0$
 small enough so that
\begin{equation*}
\left\|U-{\frac{\delta \Theta}{\delta x}}_{\left(x(t),\epsilon(x)\right)}\right\| \le \left\|U^{-1}\right\|^{-1} \ .
\end{equation*}
Therefore,the linear operator  ${\frac{\delta \Theta}{\delta x}}:C^2_0\left([0,T],\R\right)\rightarrow
C^0\left([0,T],\R\right) \times \R $ is invertible and its inverse has bounded norm
\begin{equation*} 
\left\| \left[{\frac{\delta \Theta}{\delta x}}_{\left(x(t),\epsilon(x)\right)}\right]^{-1} \right\|^{-1} \ge
\left\|U^{-1}\right\|^{-1} - \left\|U-{\frac{\delta \Theta}{\delta x}}_{\left(x(t),\epsilon(x)\right)}\right\| \ .
\end{equation*}
By Lemmas \ref{exact-solutions_1} and \ref{exact-solutions_m1}, there exists $x_\ell(v;t) \in C^2_0([0,T],\R)$ such that
$\Theta\left(x_\ell(v;t),0\right)=0$. Hence, by the implicit function theorem and by the invertibility of the
operator ${\frac{\delta \Theta}{\delta x}}$, there exist $\delta_1>0$ and $\delta_3>0$ small enough with the
following property: for all $\left|v\right|<\delta_1$ and for all
$\left\|\epsilon\right\|_{C^2([-\delta_2,\delta_2],\R)} < \delta_3$, there exists a unique solution
$x_{\epsilon}(v;t) \in C^2_0([0,T],\R)$ such that
\begin{equation} \label{x_cont}
\Theta\left(x_{\epsilon}(v;t), \epsilon(x)\right)=0 \ .
\end{equation}
Differentiating \eqref{x_cont} with respect to $\epsilon$, we get
\begin{equation*} 
{\frac{\delta \Theta}{\delta x}}_{\left(x_{\epsilon}(v;t),\epsilon(x)\right)} {\frac{\delta x}{\delta
\epsilon}}_{\left(\epsilon(x)\right)} + {\frac{\delta \Theta}{\delta
\epsilon}}_{\left(x_{\epsilon}(v;t),\epsilon(x)\right)} = 0 \ .
\end{equation*}
Hence, the operator $\frac{\delta x}{\delta \epsilon}:C^2\left(\left[-\delta,\delta\right],\R\right)\rightarrow
C^2_0\left([0,T],\R\right)  $ is given by
\begin{equation*} 
{\frac{\delta x}{\delta \epsilon}}_{\left(\epsilon(x)\right)} = -\left[{\frac{\delta \Theta}{\delta
x}}_{\left(x_{\epsilon}(v;t),\epsilon(x)\right)}\right]^{-1} {\frac{\delta \Theta}{\delta
\epsilon}}_{\left(x_{\epsilon}(v;t),\epsilon(x)\right)}  \ .
\end{equation*}
Therefore, there is $K>0$ such that
\begin{eqnarray*} 
\left\|{\frac{\delta x}{\delta \epsilon}}_{\left(\epsilon(x)\right)}\alpha\right\|_{C^2\left([0,T],\R\right)} &
= & \left\| \left[{\frac{\delta \Theta}{\delta x}}_{\left(x_{\epsilon}(v;t),\epsilon(x)\right)}\right]^{-1}
{\frac{\delta \Theta}{\delta \epsilon}}_{\left(x_{\epsilon}(v;t),\epsilon(x)\right)}
\alpha \right\|_{C^2\left([0,T],\R\right)} \nonumber \\
& \le & \left\| \left[{\frac{\delta \Theta}{\delta x}}_{\left(x_{\epsilon}(v;t),\epsilon(x)\right)}\right]^{-1}
\right\| \left\|{\frac{\delta \Theta}{\delta \epsilon}}_{\left(x_{\epsilon}(v;t),\epsilon(x)\right)} \alpha
\right\|_{C^0\left([0,T],\R\right) \times \R } \\
& \le & K\left\| \alpha\left(x_{\epsilon}(v;t)\right) \right\|_{C^0\left([0,T],\R\right) } \nonumber \ .
\end{eqnarray*}
\end{proof}

\begin{lemma} \label{v4-away}
Let $V(x)$ be a perturbed quartic potential. For every $T>0$, there exists $v_0>0$ such that, for all $0\le t \le T$ and for all $|v| <v_0$, the solution $x(v;t)$ of the Euler--Lagrange equation associated to the potential $V(x)$ satisfies the following estimate
\begin{equation*}
\left\|x(v;t)- x_{\ell}(v;t)\right\|_{C^2\left([0,T],\R\right)} < O\left(v^4 \right) \nonumber \ ,
\end{equation*}
where $x_{\ell}(v;t)$ are the solutions of the Euler--Lagrange equation associated to the potential $V_\ell(x)$ and $\ell$ is equal to the sign of $V^{(4)}(0)$.
\end{lemma}

\begin{proof}
Let us consider the equation
\begin{equation} \label{eq_e}
\overset{..}{x} + x + \ell x^3  + \epsilon(x) = 0 \ ,
\end{equation}
Let $f'(x)$ be as given by \eqref{NV}. Using Theorem \ref{Imp_arg}, there exists $v_0>0$ small enough such that,
for every $0\le k \le 1$ and every $|v|<v_0$, the equation \eqref{eq_e}, with $\epsilon(x)=k f'(x)$, has a
unique solution $x_k(v;t)$ with initial conditions $x_k(v;0)=0$ and $\overset{.}{x}_k(v;0)=v$. Furthermore, the
solution $x_k(v;t)$ is periodic and there exists $C_1>0$ such that $|x_k(v;t)| <C_1|v|$ for all $0 \le t \le T$
(see Lemma \ref{extremes}).   By \eqref{NV}, there exists $C_2>0$ such that
$$
\left\|f'\right\|_{C^0\left([-C_1|v|,C_1|v|],\R\right)} \le C_2v^4 \ .
$$
Hence, using \eqref{est_der_x}, there exists $K>0$ such that
\begin{eqnarray*}
\left\|x(v;t)- x_{\ell}(v;t)\right\|_{C^2\left([0,T],\R\right)} & \le & \sup_{0 \le k \le 1} \left\|{\frac{\delta x}{\delta \epsilon}}_{k f'(x)}f'(x)\right\|_{C^2\left([0,T],\R\right)}  \nonumber \\
 & \le & K \left\|f'\left(x_k(v;t)\right)\right\|_{C^0\left([0,T],\R\right)} \nonumber  \\
& \le & K \left\|\epsilon\left(x\right)\right\|_{C^0\left([-C_1|v|,C_1|v|],\R\right)} \nonumber \\
& \le &  K C_2 v^4 \nonumber \ .
\end{eqnarray*}
\end{proof}

\begin{proposition}
\label{lemma3} Let $V(x)$ be a perturbed quartic potential. For every $T>0$, there exists $v_0>0$
 such that, for all $0\le t \le T$ and for all $|v| <v_0$, the solution $x(v;t)$
 of the Euler--Lagrange equation associated to the potential $V(x)$
satisfies the following estimate
\begin{equation*}
x \left(v;t\right) \in  v\sin (e_{\ell}(v) t) - \frac{1}{8}v^3 g_{\ell}(e_\ell(v)t)  \pm O\left(v^{4}\right)   \ ,
\end{equation*}
where $x_{\ell}(v;t)$ are the solutions of the Euler--Lagrange equation associated to the potential $V_\ell(x)$ and $\ell$ is equal to the sign of $V^{(4)}(0)$.
In particular, we get
\begin{equation*}
x\left(v;t\right)  \in  v\sin \left( t\right)\pm O\left(|v|^3|t|+v^4\right)  \ .
\end{equation*}
\end{proposition}
\begin{proof}
The proof follows from combining Lemmas \ref{np-lemma3} and \ref{v4-away}.
\end{proof}


\section{Period map $T$}\label{period_map}

In this section,
we are going to compute the first three terms of the Taylor's expansion of the period map $T(v)$  which associates to each initial velocity $v \ne 0$ of a solution $x(v;t)$, with $x(v;0)=0$, its smallest period $T(v)$  in a small neighbourhood of $0$.

Since $x=0$ is a local minimum of the perturbed quartic potential, the solutions $x(v;t)$ of the
Euler--Lagrange equation \eqref{NEL_P}, with initial conditions $x(v;0)=0$ and $\overset{.}x(v;0)=v$, are periodic
for all small values of $v$.
Hence, there exists $v_0 > 0$, small enough, such that, for all $|v| <v_0$,
the extreme points   $x^m(v)\le 0 \le x^M(v)$    of a solution $x(v;t)$
are well-defined.
Since the energy
of the system is conserved along its orbits and the initial energy is equal to $\frac{1}{2}v^2$, we obtain that  $x^m(v)$ and $x^M(v)$    are implicitly determined
by
$$V(x^M(v))=V(x^m(v))=\frac{1}{2}v^2 \ .$$
Similarly, we will denote by $x_\ell^m(v)\le 0  \le x_\ell^M(v)$
the extreme points of the solution $x_\ell(v;t)$ of \eqref{NPEL}.
\begin{definition}
The \emph{period map} $T:(-v_0,v_0)
\longrightarrow \R $ is defined by
$$T\left( v\right) = \sqrt{2} \int_{x^{m}(v)}^{x^{M}(v)}
\left(\frac{1}{2}v^2-V(x)\right)^{-1/2} \rmd x \ .$$
\end{definition}
   We note that the period of the solution $x(0,t)$ is equal to $0$, but
we keep $T(0)=2\pi$ for the period map $T$ to be continuous and smooth.

\begin{lemma} \label{extremes}
Let $V(x)$ be a a perturbed quartic potential
and let $V_ \ell(x)$ be a quartic potential.
 There  exists $v_0 > 0$ small enough such that, for all $|v| <
v_0$, we have that $x_{\ell}^m(v)=-x_{\ell}^M(v)$, $x^{m}(v) \in O(|v|)$,  and
$$
  x^M(v) \in x_{\ell}^M(v) \pm O\left(|v|^5\right)  \text{    and      }
x^m(v) \in x_{\ell}^m(v) \pm O\left(|v|^5\right) \ .
$$
\end{lemma}

\begin{proof}
Since $V(x) \in \frac{1}{2}x^2 \pm O\left(x^4\right)$ and
$V(x^M(v))=V(x^m(v))=\frac{1}{2}v^2$, we get $x^{m}(v) \in \pm O(|v|)$ and $x^{m}(v) \in \pm O(|v|)$. Recall
that $V_\ell(x)=\frac{1}{2}x^2 + \frac{\ell}{4}x^4$. The solutions of the equation $V_\ell\left(x_\ell(v)
\right)=\frac{1}{2}v^2$ are given by
\begin{equation*}
x_\ell^M(v) = -x_\ell^m(v)  = \left(-\ell+\ell\left(1+2\ell v^2\right)^{1/2}\right)^{1/2}  \ .
\end{equation*}
By the definition of the perturbed quartic potentials in \eqref{NV}, we have that
\begin{eqnarray*}
V(x)&=&\frac{1}{2}x^2 + \frac{\ell}{4}x^4 + f(x) \nonumber \\
&\in& \frac{1}{2}x^2 + \frac{\ell}{4}x^4  \pm O\left(|x|^5\right)\nonumber
\\
&\subset & V_{\ell}(x)\pm O\left(|x|^5\right) \label{x2p} \ .
\end{eqnarray*}
Therefore, we get
\begin{eqnarray*}
V_\ell\left(x^M(v)\right) &\in & V(x^M(v))\pm O\left(|x^M(v)|^5\right)\nonumber \\
&\subset & \frac{1}{2}v^2 \pm O\left(|x|^5\right) \ .
\end{eqnarray*}
Hence, we have that
\begin{eqnarray*}
x^M(v) &\in&  \left(-\ell+\ell\left(1+4\ell\left(\frac{1}{2}v^2\pm
O(|x|^5)\right)\right)^{1/2}\right)^{1/2} \nonumber \\
&\subset & x_{\ell}^M(v) \pm O\left(|v|^5\right)  \ .
\end{eqnarray*}
Similarly, we get
\begin{eqnarray*}\label{x4.2p}
x^m(v) &\in&  -\left(-\ell+\ell\left(1+4\ell\left(\frac{1}{2}v^2\pm
O(|x|^5)\right)\right)^{1/2}\right)^{1/2} \nonumber \\
&\subset & x_{\ell}^m(v) \pm O\left(|v|^5\right)  \  .
\end{eqnarray*}
\end{proof}

\begin{lemma} \label{period-map} Let $V_\ell(x)$ be a quartic potential.
The period map $T:(-v_0,v_0) \longrightarrow \R $  satisfies the following estimate
$$
T_\ell(v)=2\pi -\frac{3\pi}{4}\ell v^2 \pm O\left(v^4\right) \label{NPT} \ ,
$$
where $\ell$ is equal to the sign of $V^{(4)}(0)$.
\end{lemma}

\begin{proof}
By Lemma \ref{extremes}, we have that
\begin{eqnarray*}
T_\ell \left( v\right) &=& \sqrt{2} \int_{x^{m}_\ell(v)}^{x^{M}_\ell(v)} \left(\frac{1}{2}v^2-V_\ell(x)\right)^{-1/2} \rmd x \nonumber \\
&=& 2\sqrt{2} \int_{0}^{x^{M}_\ell(v)} \left(\frac{1}{2}v^2-V_\ell(x)\right)^{-1/2} \rmd x \ .
\end{eqnarray*}
By the change of coordinates $z=x/x^M_{\ell}(v)$, we obtain that
\begin{equation*}
T_\ell \left( v\right) = 2\sqrt{2}x^{M}_\ell(v) \int_{0}^{1} \left(\frac{1}{2}v^2-V\left(x^{M}_\ell(v)z\right)\right)^{-1/2} \rmd z \ .
\end{equation*}
Hence,  we get
\begin{equation*}
\begin{array}{llll}
T_\ell(0)=2\pi \ , & \frac{\rmd T_\ell}{\rmd v}(0)=0 \ , & \frac{\rmd^2 T_\ell}{\rmd v^2}(0)=-\frac{3\pi\ell}{2} \ , & \frac{\rmd^3
T_\ell}{\rmd v^3}(0)=0 \nonumber
\end{array}  \ .
\end{equation*}
Therefore, we have that
\begin{equation*}
T_\ell(v)=2\pi -\frac{3\pi}{4}\ell v^2 \pm O\left(v^4\right)   \ .
\end{equation*}
\end{proof}

\begin{proposition} \label{period-map1} Let $V(x)$ be a a perturbed quartic potential.
The period map $T:(-v_0,v_0) \longrightarrow \R $  satisfies the following estimate
$$T(v)=2\pi -\frac{3\pi}{4}\ell v^2 \pm O\left(|v|^{3}\right) \ .$$
\end{proposition}

\begin{proof}
By Lemma \ref{v4-away}, there is $v_0 >0$ small enough, $T>0$ and $C >0$ with the
following properties: for all $|v| <v_0$, we have that $2T_\ell(v) <T$, and  for all
 $0 \le t \le T$, we have that
\begin{equation} \label{ffedfedf}
\left|x\left(v;t\right)-x_\ell(v;t)\right| < Cv^4 \ .
\end{equation}
By Lemma \ref{np-lemma3}, there is $K > 0$ such that, for all $|v| <v_0$, we have that
\begin{eqnarray}
x_\ell(v;2Cv^3) & \ge &  2Cv^{4} -K|v|^6 \label{PT1.2} \\
 x_\ell(v;-2Cv^3) & \le&  -2Cv^4 +K|v|^6\label{PT1.3} \ .
\end{eqnarray}
Combining \eqref{ffedfedf} and \eqref{PT1.2}, we obtain that
\begin{eqnarray}\label{PT1.4}
x\left(v;T_\ell(v)+2Cv^3\right) & \ge &
x_\ell(v;T_\ell(v)+2Cv^3) - Cv^4  \nonumber \\
& \ge & x_\ell(v;+2Cv^3) - Cv^4\nonumber \\
& \ge & Cv^{4} - K|v|^6 \ .
\end{eqnarray}
Combining \eqref{ffedfedf} and \eqref{PT1.3}, we get
\begin{eqnarray}\label{PT1.5}
x\left(v;T_\ell(v)-2Cv^3\right) & \le &
x_\ell(v;T_\ell(v)-2Cv^3) + Cv^4  \nonumber \\
& \le & x_\ell(v;-2Cv^3) + Cv^4\nonumber \\
& \le & -Cv^{4} + K|v|^6 \ .
\end{eqnarray}
Take  $v_1 < v_0$ such that $Cv^{4}  > K|v|^6$. Combining \eqref{PT1.4} and \eqref{PT1.5}, for all $|v| <v_1$ we
obtain that
$$
x\left(v;T_\ell(v)-2Cv^3\right) < 0 <x\left(v;T_\ell(v)+2Cv^3\right) \ .
$$
Therefore, by continuity of $x\left(v;t\right)$,    we have that
$$T_\ell(v)-2Cv^3 < T(v)<T_\ell(v)+2Cv^3 \ .$$
Hence, by Lemma \ref{period-map}, we get
$$
T(v)=2\pi -\frac{3\pi}{4}\ell v^2 \pm O\left(|v|^3\right) \ .
$$
\end{proof}


\section{The scaling parameter} \label{scaling_parameter}

In this section, we will use the period map $T(v)$ to determine estimates for the \emph{$(n,t)$-scaling velocity} $\gamma_{n,t}$ with relevance for the determination of the $(n,t)$-scaling parameter $\Gamma_{n,t}$ introduced in \eqref{gamma}. Furthermore, we will compute some estimates for $x(\Gamma_{n,t} v; t)$ which we be useful for the study of the \emph{$n$-renormalized trajectories} $x_n(v;t)=(-1)^n \Gamma_{n,t}^{-1} x\left(\Gamma_{n,t} v; n\pi-\ell t\right)$.

The  \emph{$(n,t)$-scaling velocity} $\gamma_{n,t} \ge 0$ has the property that there exists $K_1>0$ such that
\begin{equation*}
x(\gamma_{n,t} ; n\pi -\ell t) = 0
\end{equation*}
for every $0<t< K_1 n$. Furthermore, we have that $x(\gamma_{n-1,t-\ell\pi} ; n\pi -\ell t) = 0$.

\begin{proposition}\label{lemma4}
Let $V(x)$ be a perturbed quartic potential and $\ell\in\{-1,1\}$ be equal to the sign of $V^{(4)}(0)$.
There exists $K>0$ such that for every $n\ge 1$ and for every $t \in[0,K n]$ the $(n,t)$-scaling velocity $\gamma_{n,t}>0$ given by
\begin{equation}\label{4.1}
T\left(\gamma_{n,t}\right) =2\pi - \frac{2\ell t}{n}
 \end{equation}
is well-defined.  Furthermore, the following estimates are satisfied
\begin{eqnarray}
{\gamma_{n,t}}^{2} &\in & \frac{8 t}{3\pi n} \pm O\left(\left|\frac{t}{n}\right|^{3/2}\right)  \label{4.3} \\
\frac{n}{2}T\left(\gamma_{n,t} v\right) & = & n\pi - \ell t -R_{\ell,v,t} \pm O\left(\frac{|t|^{3/2}}{n^{1/2}}\right) \nonumber  \ ,  
\end{eqnarray}
where $R_{\ell,v,t}=- \ell t \left(1 -v^2 \right)$. In particular, we have that $\gamma_{n,t} \in O\left(\left(|t|/n\right)^{1/2}\right)$.
\end{proposition}

\begin{proof}
For simplicity of notation, we will denote $\gamma_{n,t}$ by $\gamma$ throughout the proof.  By
Proposition \ref{period-map1}, there is $v_0=v_0(V)>0$ such that, for all $|v|<v_0$, the period map $T(v)$
satisfies
\begin{equation} \label{TPl}
T(v)=2\pi -\frac{3\pi}{4} \ell v^2 \pm O\left(|v|^3\right) \ .
\end{equation}
Hence, we have that
\begin{equation} \label{a43}
T\left(\gamma\right)-2\pi= -\frac{3\pi}{4} \ell {\gamma}^2 \pm O\left(\left|\gamma\right|^3\right) \ .
\end{equation}
By \eqref{4.1}, we get
\begin{equation}  \label{a41}
T\left(\gamma\right)-2\pi = -\frac{2\ell t}{n} \ .
\end{equation}
Combining \eqref{a41} and \eqref{a43}, we obtain
\begin{equation*}
-\frac{2\ell t}{n}  = -\frac{3\pi}{4} \ell {\gamma}^2 \pm O\left(\left|\gamma\right|^3\right) \ ,
\end{equation*}
Therefore, there is $K >0$ such that $\gamma=\gamma_{n,t}$ is well--defined, for every $0<t< Kn$, and
 ${\gamma}^{2} \in O\left( \frac{t}{n}\right)$. Furthermore,
\begin{eqnarray*} 
{\gamma}^{2} &\in & \frac{8 t}{3\pi n} \pm O\left( \left|\frac{t}{n}\right|^{3/2}\right)  \ .
\end{eqnarray*}
Hence, by \eqref{TPl}, we get
\begin{eqnarray*} 
T\left(\gamma\right) -T\left( \gamma v\right) &\in & -\frac{3\pi}{4}\ell \left({\gamma}^2 -{\gamma}^2 v^2 \right) \pm O\left(\left|\frac{t}{n}\right|^{3/2}\right) \nonumber \\
&\subset & -\frac{3\pi}{4} \ell {\gamma}^2 \left(1 -v^2 \right) \pm O\left(\left|\frac{t}{n}\right|^{3/2}\right)\nonumber \\
&\subset & -2\ell \frac{t}{n}\left(1 -v^2 \right) \pm O\left( \left|\frac{t}{n}\right|^{3/2}\right)  \ .
\end{eqnarray*}
Thus, by \eqref{4.1}, we have that
\begin{eqnarray*} 
n\pi - \ell t &=& \frac{n}{2}T\left(\gamma\right) \nonumber \\
&\in &\frac{n}{2}T\left(\gamma v\right) - \ell t \left(1 -v^2 \right) \pm O\left( \frac{|t|^{3/2}}{n^{1/2}}\right) \  .
\end{eqnarray*}
\end{proof}

\begin{remark}
Note that inequalities \eqref{4.3} can be restated as
\begin{equation}\label{q_gamma_rel}
{\gamma_{n,t}}^{2} \in  {\Gamma_{n,t}}^{2} \pm O\left(\left|\frac{t}{n}\right|^{3/2}\right) \ .
\end{equation}
\end{remark}

\begin{lemma}\label{gamma_per}
Let $V(x)$ be a perturbed quartic potential and $\ell\in\{-1,1\}$ be equal to the sign of $V^{(4)}(0)$.
There exists $K>0$ such that for every $n\ge 1$ and for every $t \in[0,K n]$ the $(n,t)$-scaling parameter $\Gamma_{n,t}$ satisfies
\begin{equation*}
n\pi - \ell t = \frac{n}{2}T\left(\Gamma_{n,t}~ v\right) + R_{\ell,v,t} \pm O\left(\frac{|t|^{3/2}}{n^{1/2}}\right)\ ,
\end{equation*}
where $R_{\ell,v,t}=- \ell t \left(1 -v^2 \right)$.
\end{lemma}
\begin{proof}
For simplicity of notation, we will denote $\gamma_{n,t}$ by $\gamma$ and $\Gamma_{n,t}$ by $\Gamma$ throughout the proof. By Lemma \ref{lemma4}, we get that
\begin{equation}\label{gamma_per_1} 
n\pi - \ell t =  \frac{n}{2}T\left(\gamma v\right) - \ell t \left(1 -v^2 \right) \pm O\left( \frac{|t|^{3/2}}{n^{1/2}}\right) \ .
\end{equation}
From Proposition \ref{period-map1}, we obtain that
\begin{equation}\label{gamma_per_2} 
T\left(\gamma v\right) \in  2\pi -\frac{3\pi}{4}\ell \gamma^2 v^2 \pm O\left(\gamma^{3}\right) \ .
\end{equation}
Putting together identities \eqref{q_gamma_rel} and \eqref{gamma_per_2} and Proposition \ref{period-map1}, we get
\begin{eqnarray}\label{gamma_per_3} 
T\left(\gamma v\right) & \in &  2\pi -\frac{3\pi}{4}\ell \Gamma^2 v^2 \pm O\left(\Gamma^{3}\right) \nonumber \\
& \subset &  T\left(\Gamma v\right) \pm O\left(\Gamma^{3}\right) \ .
\end{eqnarray}
Combining \eqref{gamma_per_1} and  \eqref{gamma_per_3} and recalling that $\Gamma \in O\left((|t|/n)^{1/2}\right)$, we obtain 
\begin{equation*}
n\pi - \ell t =  \frac{n}{2}T\left(\Gamma v\right) - \ell t \left(1 -v^2 \right) \pm O\left( \frac{|t|^{3/2}}{n^{1/2}}\right) \ ,
\end{equation*}
completing the proof.
\end{proof}

\begin{lemma}
\label{lemma333} Let $V(x)$ be a  perturbed quartic potential. There exists $K > 0$ such that the solution $x\left(\Gamma_{n,t}v;t\right)$ of the Euler--Lagrange equation associated to the potential $V(x)$ satisfies the following estimate
\begin{equation*}
x\left(\Gamma_{n,t}v;t\right) \in  \Gamma_{n,t}v\sin \left( t\right)\pm O\left(|\Gamma_{n,t}|^3|t|\right)  \ ,
\end{equation*}
for all $v \in [-1,1]$ and for all $t \in [0, Kn^{3/5}]$.
\end{lemma}

\begin{proof}
For simplicity of notation, we will denote $\Gamma_{n,t}$ by $\Gamma$ trough out the proof. Let $x \ge 0$ be such that $t=xT\left(\Gamma v\right)$, and denote by $[x]$ the fractional part of $x$. By Proposition \ref{lemma4}, we have that $x=O(t)$ and
\begin{eqnarray*} \label{sdfsrg7}
t& = &  (x-[x])T\left(\Gamma v\right)+[x]T\left(\Gamma v\right)  \nonumber\\
 & \in & 2 \pi (x-[x]) +[x]T\left(\Gamma v\right)  \pm O \left(t\Gamma^2\right)  \ .
\end{eqnarray*}
Hence, we get
\begin{equation} \label{sdfsrg1}
[x]T\left(\Gamma v\right)   \in  t- 2 \pi n \pm O \left(t\Gamma^2\right)  \ .
\end{equation}
By Proposition \ref{lemma3}, we have that
\begin{eqnarray*}
x\left(\Gamma v;t\right) & = &  x\left(\Gamma v; (x-[x])T\left(\Gamma v\right)+[x]T\left(\Gamma v\right) \right) \nonumber \\
& = &  x\left(\Gamma v; [x]T\left(\Gamma v\right) \right) \nonumber \\
& \in &  \Gamma v \sin\left([x]T\left(\Gamma v\right)\right)  \pm O \left(\Gamma^3\right)  \ . 
\end{eqnarray*}
Hence, by \eqref{sdfsrg1}, we obtain that
\begin{eqnarray*}
x\left(\Gamma v;t\right)  & = &    \Gamma v \sin([x]T\left(\Gamma v\right) )  \pm O \left(\Gamma^3 \right)  \nonumber \\
& \in &  \Gamma v \sin\left(t-2 \pi (x-[x]) \pm O \left(t\Gamma^2\right)\right)  \pm O\left(\Gamma^3\right)  \\
& \subset &  \Gamma v \sin \left(t  \pm O \left(t\Gamma^2\right)   \right)  \pm O \left(\Gamma^3\right)      \\
& \subset &  \Gamma v \sin(t    )  \pm O \left(t\Gamma^3\right)      \ .
\end{eqnarray*}
Choosing $K>0$ small enough, for all $t \in [0, Kn^{3/5}]$,
we have that $O \left(t\Gamma^3\right) \le O\left(t^{5/2}/n^{3/2} \right) < 1$.
\end{proof}


\section{Asymptotic universality of trajectories} \label{trajectories}

In this section we combine the results of the previous sections to prove Proposition \ref{main}, i.e.  the $n$-renormalized trajectories
$x_n(v;t)$ converge to universal asymptotic trajectories $X_\ell(v;t)$ introduced in section \ref{intro_results}.

\begin{proof}[Proof of Proposition \ref{main}]
For simplicity of notation, we will denote $\Gamma_{n,t}$ by $\Gamma$ trough out the proof.
By Lemma \ref{gamma_per}, we have that
\begin{equation*}
x\left(\Gamma v; n\pi -\ell t\right) = x\left(\Gamma v;\frac{n}{2}T\left(\Gamma v\right) + R_{\ell,t,v} \pm O \left( q^3 n \right)\right)  \ .
\end{equation*}
Since $T\left( \Gamma v\right)$ is the period of $x\left(\Gamma v ; t\right)$, we get
\begin{equation*} 
x\left(\Gamma v; n\pi -\ell t\right) = x\left(\Gamma v;\left[\frac{n}{2}\right]T\left(\Gamma v\right) + R_{\ell,v,t} \pm O \left( \Gamma^3 n \right)\right) \ ,
\end{equation*}
where $\left[n/2\right]$ denotes the fractional part of $n/2$. By Proposition \ref{period-map1}, we have  that  $T\left( \Gamma v\right)=2\pi \pm O\left(\Gamma^2\right)$. Hence,
we have that
\begin{equation*} 
x\left(\Gamma v; n\pi -\ell t\right) = x\left(\Gamma v ; 2\pi\left[\frac{n}{2}\right] + R_{\ell,v,t} \pm O \left(\Gamma^2+\Gamma^3 m \right) \right) \ .
\end{equation*}
By Proposition \ref{lemma4}, we have that $R_{\ell,v,t} \in O(t)$. Hence, by Lemma \ref{lemma333}, we get
\begin{equation}
x\left(\Gamma v ; n\pi -\ell t\right) \in (-1)^{n} \Gamma v\sin\left( R_{\ell,v,t} \pm O\left(\Gamma^3 n+ \Gamma^2 \right)\right) \pm O\left(\Gamma^3 \left(\Gamma^2+\Gamma^3 n+t \right)  \right) \label{x-conv0.1}  \ .
\end{equation}
By Proposition \eqref{gamma}, we have  $\Gamma \in O\left(\left(|t|/n\right)^{1/2}\right)$. Hence, we fix  $K>0$ small enough so that
 $O\left(\Gamma^3 n + \Gamma^2 \right) < 1$ for every $t \in [0, K n^{1/3}]$. Thus, using the trigonometric formulas, we obtain that
\begin{equation}
\sin\left(R_{\ell,v,t} \pm O\left(\Gamma^3 n+ \Gamma^2 \right) \right) \in \sin\left( R_{\ell,v,t} \right) \pm O\left(\Gamma^2+\Gamma^3 n \right) \label{x-conv0.2} \ .
\end{equation}
Combining \eqref{x-conv0.1} and \eqref{x-conv0.2}, we obtain
\begin{equation*}
x\left(\Gamma v ; n\pi -\ell t \right) \in (-1)^{n}\Gamma v\sin\left(R_{\ell,v,t}\right) \pm O\left( \Gamma^3 + \Gamma^3 t+ \Gamma^4 n \right)  \ .
\end{equation*}
Therefore,  we get
\begin{eqnarray*}
(-1)^{n}\Gamma^{-1}x\left(\Gamma v ; n\pi -\ell t\right) & \in & v\sin\left(R_{\ell,v,t}\right)\pm O\left(\Gamma^3 n+ \Gamma^2 + \Gamma^2 t\right) \nonumber \\
& \subset & X_\ell(v;t) \pm O\left(  \Gamma^2 + \Gamma^2 t + \Gamma^3 n \right) \ .
\end{eqnarray*}
Therefore, we obtain that
\begin{equation*}
(-1)^{n}\Gamma^{-1}x\left(\Gamma v ; n\pi -\ell t\right)   \in  X_\ell(v;t) \pm O\left(  \frac{|t|}{n} + \frac{|t|^2}{n} + \frac{|t|^{3/2}}{n^{1/2}} \right) \ .
\end{equation*}
\end{proof}

The next result follows easily from Proposition \ref{main}. It is worth to point out that it is equivalent to Theorem \ref{corolario_traj_q2_intro}.

\begin{corollary} \label{corolario_traj_x2}
Let $V(x)$ be a perturbed quartic potential and let $K>0$ be as in Proposition \ref{main}. For every $0 < \epsilon < 1/3$, we have that
\begin{equation*} 
\|x_n(v;t) - X_\ell (v;t) \|_{C^0(\left[ -1,1\right] \times [0, K n^{1/3-\epsilon}],\R)} < O\left( {n^{-3\epsilon/2 }}\right) \ ,
\end{equation*}
where $\ell$ is equal to the sign of $V^{(4)}(0)$.
\end{corollary}


\section{Conclusions} \label{conclusion}

We have studied the dynamics of a family of mechanical systems that includes the pendulum at small neighbourhoods of an elliptic equilibrium and characterized such dynamical behaviour through a renormalization scheme. We have proved that the asymptotic limit of the renormalization scheme introduced in this paper is universal: it is the same for all the elements in the considered class of mechanical systems. As a consequence we have obtained an universal asymptotic focal decomposition for this family of mechanical systems.

We believe that the existence of an universal asymptotic focal decomposition might be useful not only on the theory of boundary value problems of ordinary differential equations but also on several distinct fields of the physical sciences such as quantum statistical mechanics, optics, general relativity and even tsunami formation. Our belief on such applications is based on the relevance that the concept of focal decomposition may have on the study of caustic formation by focusing wavefronts, of such significance to those fields.


\section*{Acknowledgments}

We thank Robert MacKay and Michael Berry for helpful discussions. 

We thank the Calouste Gulbenkian Foundation, PRODYN-ESF, POCTI, and POSI by FCT and Minist\'erio da Ci\^encia, Tecnologia e Ensino Superior, Centro de
Matem\'atica da Universidade do Minho, Centro de Matem\'atica da Universidade do Porto and CEMAPRE for their financial support. 

We also thank CNPq, FAPERJ, and FUJB for financial support.

Part of this research was done during visits by the authors to IMPA (Brazil), The University of Warwick (United Kingdom), IHES (France), CUNY (USA), SUNY (USA) and MSRI (USA), who thank them for their hospitality. 

D. Pinheiro's research was supported by FCT - Funda\c{c}\~ao para a Ci\^encia e Tecnologia through the grant with reference SFRH / BPD / 27151 / 2006 and the program ``Ci\^encia 2007''.


\bibliography{CPPP_arXiv}
\bibliographystyle{unsrt} 

\end{document}